\documentclass[11pt]{amsart}

\usepackage{amsmath,amscd,amssymb}
\usepackage[all,2cell,dvips,cmtip,matrix,arrow]{xy}

\newtheorem{theorem}{Theorem}[section]

\newtheorem{proposition}[theorem]{Proposition}

\newtheorem{lemma}[theorem]{Lemma}
\theoremstyle{remark}
\newtheorem{remark}[theorem]{Remark}
\newtheorem{definition}[theorem]{Definition}
\newtheorem{remarks}[theorem]{Remarks}
\newtheorem{example}[theorem]{Example}

\newcommand\da{\dasharrow}
\newcommand\A{\mathcal{A}}

\newcommand\B{\mathbb{B}}
\newcommand\be{\begin{equation}\label}
\newcommand\ee{\end{equation}}
\newcommand\M{\mathcal{M}}
\renewcommand\L{\mathcal{L}}
\newcommand{\K}{\mathbb{K}}

\renewcommand{\O}{\on{O}}

\newcommand{\Co}{\mathcal{C}}

\newcommand{\U}{\on{U}}

\newcommand{\F}{\mathcal{F}}
\newcommand{\E}{\mathcal{E}}

\newcommand{\R}{\mathbb{R}}
\newcommand{\C}{\mathbb{C}}

\newcommand{\Z}{\mathbb{Z}}

\newcommand{\pr}{\on{pr}}
\newcommand{\Cl}{{\C \on{l}}}
\newcommand\lie[1]{\mathfrak{#1}}

\newcommand{\g}{\lie{g}}

\renewcommand{\t}{\lie{t}}
\newcommand{\Alc}{\lie{A}}
\newcommand{\on}{\operatorname}

\newcommand{\Ad}{ \on{Ad} }

\newcommand{\End}{ \on{End} } 

\newcommand{\id}{\on{id}}
 
\renewcommand{\ker}{ \on{ker}} 
 \newcommand{\Spin}{ \on{Spin}}
\newcommand{\SU}{ \on{SU}} 
\newcommand{\SO}{ \on{SO}}

\newcommand{\Mult}{ \on{Mult}}

\newcommand{\cone}{ \on{cone} }
\newcommand{\D}{ \mathcal{D} }
\newcommand\dirac{/\kern-1.2ex\partial} % Dirac operator
\newcommand\qu{/\kern-.7ex/} % Categorical quotients
\newcommand{\fus}{\circledast}  % Level
\newcommand{\Waff}{W^{\on{aff}}} % affine Weylgroup 
\newcommand{\lra}{\longrightarrow}
\newcommand{\hra}{\hookrightarrow}
\newcommand{\xra}{\xrightarrow}

\renewcommand{\d}{{\mbox{d}}}
\newcommand{\ol}{\overline}

\newcommand\Phinv{\Phi^{-1}}

\newcommand\Sig{\Sigma}
\newcommand\sig{\sigma}

\newcommand\Om{\Omega}
\newcommand\om{\omega}

\newcommand{\f}{\frac}

\renewcommand{\H}{\ca{H}}

\renewcommand{\l}{\langle}
\renewcommand{\r}{\rangle}
\newcommand{\hh}{{\textstyle \f{1}{2}}}

\newcommand\pt{\on{pt}}

\newcommand\vol{\on{vol}}

\newcommand\Ch{\on{Ch}}

\newcommand\beqn{\begin{equation}}      
\newcommand\eeqn{\end{equation}}      
\newcommand{\ca}{\mathcal}
\newcommand{\wh}{\widehat}
\newcommand{\wt}{\widetilde}
\newcommand{\mf}{\mathfrak}
\newcommand{\beq}{\begin{eqnarray*}}
\newcommand{\eeq}{\end{eqnarray*}}

\newcommand{\cox}{\mathsf{h}^\vee}

\newcommand{\red}{{\on{red}}}
\newcommand{\reg}{{\on{reg}}}
\newcommand{\op}{{\on{op}}}

\begin{document}

\title[]{Twisted $K$-homology and group-valued moment maps}

\vskip.2in
\author{E. Meinrenken}
%\address{University of Toronto, Department of Mathematics,
%40 St George Street, Toronto, Ontario M4S2E4, Canada }
%\email{mein@math.toronto.edu}
%
\date{\today}
\begin{abstract}
Let $G$ be a compact, simply connected Lie group. We develop a `quantization functor' from 
pre-quantized quasi-Hamiltonian $G$-spaces $(M,\omega,\Phi)$ at level $k$ to the fusion ring (Verlinde algebra)
$R_k(G)$. The quantization $\ca{Q}(M)\in R_k(G)$ is defined as a push-forward in twisted equivariant 
$K$-homology. It may be computed by a fixed point formula, similar to the equivariant index theorem for 
$\Spin_c$-Dirac operators. Using the formula, we calculate $\ca{Q}(M)$ in several examples. 
\end{abstract}
\maketitle 
\setcounter{tocdepth}{2}
%%%%%%%%%%%%%%%%%%%%%%%%%%%%%%%%%%%%%%%%%%%%%%%%%%%%%%%%%%%%%%%%%%%%%%%%%%%%%%%%%
%%%%%%%%%%%%%%%%%%%%%%%%%%%%%%%%%%%%%%%%%%%%%%%%%%%%%%%%%%%%%%%%%%%%%%%%%%%%%%%%%

\section{Introduction}
Let $G$ be a compact Lie group, and $M$ a compact
symplectic $G$-manifold. The symplectic form $\om_0$ determines 
an equivariant $\Spin_c$-structure, with corresponding Dirac operator $\dirac$. If the action is
Hamiltonian, with moment map $\Phi_0$, and $L\to M$ is an equivariant
pre-quantum line bundle (where the lift of the $G$-action is
determined by $\Phi_0$), one can consider the Dirac operator $\dirac_L$ with coefficients in $L$. 
The equivariant index 
\[ \ca{Q}(M):=\on{index}_G(\dirac_L)\in R(G),\]
will be called the \emph{quantization} of the Hamiltonian $G$-space $(M,\om_0,\Phi_0)$, following \cite{gu:re}. The quantization has a
number of nice properties. 
\begin{enumerate}
\item[(i)] For products, $\ca{Q}(M_1\times
M_2)=\ca{Q}(M_1)\ca{Q}(M_2)$.
\item[(ii)] Let $M^*$ denote $M$ with
the opposite symplectic structure $\om_0^*=-\om_0$ and moment map
$\Phi_0^*=-\Phi_0$. Then $\ca{Q}(M^*)=\ca{Q}(M)^*$. \item[(iii)] If $\mathcal
O\subset \g^*$ is an integral coadjoint orbit, then $\ca{Q}(\mathcal
O)$ is the character of the irreducible representation labeled by
$\mathcal O$. \item[(iv)] \emph{Quantization commutes with
reduction}: The multiplicity of the trivial representation in
$\ca{Q}(M)$ equals $\ca{Q}(M\qu G)$, the quantization of the
symplectic quotient.  (See \cite{me:si} for the precise statement for singular quotients.)
\end{enumerate}
 For any $g\in G$, the Atiyah-Segal-Singer theorem gives a
formula for $\ca{Q}(M)(g)$ as a sum of integrals over the fixed point
manifolds of $g$.

In this paper we will develop an analogous quantization for
\emph{quasi-Hamiltonian $G$-spaces} $(M,\om,\Phi)$. 
The moment map
\[\Phi\colon M\to G\] 
of a q-Hamiltonian $G$-space takes values in the group itself rather
than in the dual of the Lie algebra. The moment map condition for an 
ordinary Hamiltonian $G$-space, together with the condition that $\om$ be closed, are 
replaced by the property (using the Cartan model of equivariant de Rham theory, cf.~ Appendix \ref{app:eq}),  
\begin{equation}\label{eq:condi}
 \d_G\om=-\Phi^*\eta_G.
 \end{equation}
Here $\eta_G\in\Om^3_G(G)$ is the equivariant Cartan 3-form determined by an invariant inner product on the Lie algebra $\g$. The
condition that $\om$ be non-degenerate is replaced by the weaker
condition that $\ker(\om)\cap \ker(\d\Phi)=0$ everywhere. 
Products, conjugates, and reductions are defined similar to the
Hamiltonian case. Basic examples of a q-Hamiltonian space are 
conjugacy classes $\mathcal C\subset G$, with moment map the inclusion.  Another 
important example is $M=G^{2h}$ with moment map
$\Phi(a_1,b_1,\ldots,a_h,b_h)=\prod_{i=1}^h a_i b_i a_i^{-1} b_i^{-1}$
(and a certain complicated formula for $\om$): the symplectic quotient
$G^{2h}\qu G$ is the moduli space of flat connections on a surface of
genus $h$.

Many of the results of this paper are developed under the assumption that $G$ is 
compact, connected, and semi-simple. For the purposes of this introduction let us also assume that 
$G$ is simple and simply connected. 
For any non-negative integer $k$, let $R_k(G)$ be the \emph{level $k$ fusion ring} (Verlinde algebra) of $G$. It is a quotient of the representation ring $R(G)$ by the level $k$ fusion ideal $I_k(G)$, defined as follows. 
Fix a maximal torus $T$ and a fundamental Weyl chamber for $G$. Use the basic inner product on $\g$ to identify the Lie algebra with its dual. Let $\Lambda^*_k$ be the set of \emph{level $k$ weights}, i.e.~ weights $\lambda\in \t^*=\t$ such that $\f{1}{k}\lambda$ is in the fundamental Weyl alcove. Let $\cox$ be the dual Coxeter number and $\rho$ the half-sum of positive roots. Then $I_k(G)$ is the ideal of characters vanishing at all points
\begin{equation}\label{eq:tlambda1}
 t_\lambda=\exp({\textstyle \f{\lambda+\rho}{k+\cox}}),\ \ \lambda\in\Lambda^*_k.\end{equation}
(For details, see Appendix \ref{app:fus}.)

%Then the integral cohomology of $G$ vanishes in degree
%$1,2$, and is equal to $\Z$ in degree $3$. 
%
The q-Hamiltonian $G$-space $(M,\om,\Phi)$ is \emph{pre-quantizable}
if the class in relative cohomology $H^3_G(\Phi)$ defined by the pair $(\om,\eta_G)$ is integral.  
In particular, $\eta_G$ defines a class of $H^3_G(G,\Z)=\Z$; the integer $k$ defined by this 
isomorphism is the \emph{level}. We will work with the geometric realization of integral degree $3$ 
cohomology classes in terms of \emph{Dixmier-Douady bundles}.  (See Section \ref{sec:dd} for background on Dixmier-Douady theory and twisted $K$-homology.) Let $\A\to G$ be a $G$-equivariant Dixmier-Douady bundle whose Dixmier-Douady class is 
a generator of $H^3_G(G,\Z)=\Z$, and let $\A^k$ be its $k$-th power. Then the equation 
\eqref{eq:condi} is pre-quantized to a $G$-equivariant Morita trivialization 
\[ (\Phi,\E)\colon (M,\C)\da (G,\A^k);\]
this morphism is the counterpart to a pre-quantum line bundle. In \cite{al:ddd}, we obtained a 
distinguished $G$-equivariant Morita morphism 
\[ (\Phi,{\ca{S}})\colon (M,\Cl(TM))\da (G,\A^{\cox})\]
from the Clifford bundle of $M$. This morphism is the counterpart of the $\Spin_c$-structure of a symplectic manifold; 
indeed the spinor module for the $\Spin_c$-structure can be regarded as a Morita trivialization $(M,\Cl(TM))\da 
(\pt,\C)$. Tensoring,  one obtains a Morita morphism 
$(\Phi,{\ca{S}}\otimes\E)\colon (M,\Cl(TM))\da (G,\A^{k+\cox})$, and a resulting 
push-forward in twisted equivariant $K$-homology
\[ \Phi_*\colon K_0^G(M,\Cl(TM))\to K_0^G(G,\A^{k+\cox}).\]
A theorem of Freed-Hopkins-Teleman \cite{fr:lo1} identifies the $K$-homology group on the right hand side 
with $R_k(G)$. The $K$-homology group on the left hand side contains a distinguished element $[M]$, the 
\emph{fundamental class} of $M$. We define the quantization to be its push-forward,
\[ \ca{Q}(M):=\Phi_*[M]\in R_k(G).\] 
We will find that the properties of this quantization procedure are similar to those 
for ordinary Hamiltonian spaces: 
\begin{enumerate}
\item[(i)] For fusion products, $\ca{Q}(M_1\times
M_2)=\ca{Q}(M_1)\ca{Q}(M_2)$. 
\item[(ii)] Let $M^*$ denote $M$ with
the opposite 2-form $\om^*=-\om$ and moment map
$\Phi^*=\Phi^{-1}$. Then $\ca{Q}(M^*)=\ca{Q}(M)^*$. 
\item[(iii)] If $\Co\subset G$ is a conjugacy class with a level $k$ pre-quantization, 
then $\ca{Q}(\Co)$ is the character of the basis element of $R_k(G)$ 
labeled by $\Co$. \item[(iv)] \emph{Quantization commutes with
reduction}: The multiplicity of the trivial representation in
$\ca{Q}(M)$ equals $\ca{Q}(M\qu G)$, the quantization of the
symplectic quotient.
\end{enumerate}
$\ca{Q}(M)$ may be computed by localization: Let $t=t_\lambda$ be one of the elements
\eqref{eq:tlambda1}. The map $R(G)\to \C$ given by the evaluation at $t$ descends to $R_k(G)$, and
hence the number $\ca{Q}(M)(t)\in\C$ is defined. By equivariance of $\Phi$, and since $t$ is regular, the fixed
point set $M^t$ maps to $G^t=T$. It turns out that the restriction 
of $\A^{k+\cox}$ to $T$ is $\langle t\rangle$-equivariantly Morita trivial, where $\langle t\rangle$ is the 
finite subgroup generated by $t$. Once we fix such a Morita trivialization, 
one obtains a $\langle t\rangle$-equivariant Morita trivialization of 
$\Cl(TM)$ along $M^t$. Equivalently, the restriction of $TM$ to the fixed point set acquires
a $\langle t\rangle$-equivariant $\Spin_c$-structure. We will find that $\ca{Q}(M)(t)$ is given by
the usual Atiyah-Segal-Singer fixed point formula, as a sum over the components of the fixed point set,
\[ \ca{Q}(M)(t)=\sum_{F\subset M^t}\int_F \f{\wh{A}(F)\on{Ch}(\ca{L}_F,t)^{1/2}}{D_\R(\nu_F,t)}\]
using the $\Spin_c$-structure along the fixed point sets $F\subset M^t$. Thus, even though our quantization procedure does not 
involve a globally defined Dirac operator, the localization formula has the appearance of an equivariant index 
formula. 

In an earlier paper \cite{al:fi}, the fixed point formula was implicitly used as a `\emph{definition}' of 
the quantization of a q-Hamiltonian space, motivated by formal computations for a Hamiltonian loop group space 
associated to $(M,\om,\Phi)$. See also Carey-Wang \cite{car:fu} for a similar approach. 
The presentation of $\ca{Q}(M)$ as a $K$-homology push-forward is more satisfying conceptually: For instance, it is not obvious that the right hand side of the fixed point formula defines an element of $R_k(G)$ (rather than just 
$R_k(G)\otimes\C$). 
 
An application of our theory to a product $M=D(G)^{h}\times\Co_1\cdots\cdots \times\Co_r$ of $h$ copies of the `double'
$D(G)=G\times G$ (with moment map the group commutator) with $r$ pre-quantized conjugacy classes gives the 
Verlinde formulas for the moduli space of flat connections on a surface of genus $h$ with $r$ boundary 
components. In a forthcoming paper, we will use the formula to obtain Verlinde-type formulas for 
non-simply connected groups. 
% As a preview towards these general results, we will work out the quantization $\ca{Q}(D(\on{PU}(n)))$ of 
% the double of   $\on{PU}(n)=\SU(n)/\Z_n$, viewed as a q-Hamiltonian $\SU(n)$-space, 
% for the case that $n$ is a prime number.  

Let us finally remark that q-Hamiltonian $G$-spaces are closely related to the concept of \emph{D-branes on group manifolds} from  string theory (see e.g. Carey-Wang \cite{car:fu}). The implications of this relationship are deserving of further study. 
\vskip.2in

\noindent{\bf Acknowledgments.} This paper was written over a rather long period, and owes to discussions 
with many people. Some of the basic ideas go back to joint work with Anton Alekseev. The idea of quantizing q-Hamiltonian spaces as $K$-homology pushforwards, rather than trying to construct an operator, originated from discussions with Greg Landweber in October 2005. It is a pleasure to thank Nigel Higson, John Roe and Jonathan Rosenberg for discussions and for their generous help with aspects of equivariant $K$-homology.

\tableofcontents

\section{Twisted equivariant $K$-homology}\label{sec:dd}
In this Section, we give a quick review of twisted equivariant
$K$-homology. We begin by discussing the geometric realization of
$H^3_G(X,\Z)$ in terms of equivariant Dixmier-Douady bundles, similar to the geometric realization 
of $H^2_G(X,\Z)$ in terms of equivariant line bundles. 
\subsection{Dixmier-Douady theory}
A well-known result of Serre (see Donovan-Karoubi \cite{don:gr}) asserts that bundles of matrix algebras 
(Azumaya bundles) over a space $X$ are classified, up to stable (or Morita) isomorphism, by the torsion subgroup 
of $H^3(X,\Z)$. To incorporate non-torsion classes, it is necessary to consider matrix algebras 
`of infinite rank'. 
\subsubsection{Dixmier-Douady bundles}
In this paper, all Hilbert spaces are taken to be complex, separable Hilbert spaces.  
For any Hilbert space $\H$, we denote by $\K(\H)$ the $C^*$-algebra 
of compact operators, i.e.~ the norm closure of the finite rank 
operators inside the algebra $\B(\H)$ of bounded linear operators. 
A \emph{Dixmier-Douady bundle} $\A\to X$ (in short, \emph{DD bundle}) is a bundle of 
$C^*$-algebras, with typical fiber $\K(\H)$ and 
structure group $\on{Aut}(\K(\H))=\on{PU}(\H)$ for some Hilbert space $\H$.
(Here $\on{PU}(\H)=\on{U}(\H)/\U(1)$ carries the strong topology.) 
Isomorphisms, pull-backs, and tensor products of such bundles are
defined in the obvious way. For any $\A$, the opposite algebra bundle
$\A^\op$ (with the same vector bundle structure, but opposite 
multiplication) is a DD bundle modeled on the compact operators on the 
conjugate Hilbert space $\ol{\H}$. A \emph{Morita trivialization} of $\A\to X$ is a bundle $\E\to X$ of 
Hilbert spaces together with an isomorphism
\[ \A\to \K(\E).\] 
Equivalently, letting $\ca{P}\to X$ be the principal
$\on{PU}(\H)$-bundle associated to $\A$, a Morita trivialization
amounts to a lift of the structure group to $\on{U}(\H)$. The obstruction to the existence
of a Morita trivialization is the \emph{Dixmier-Douady class}
$\on{DD}(\A)\in H^3(X,\Z)$. 

\subsubsection{$\Z_2$-graded DD bundles}
More generally, we will need to consider $\Z_2$-graded DD bundles $\A\to X$, modeled on $\K(\H)$ for $\Z_2$-graded 
Hilbert spaces $\H$. 
The obstruction to the existence of a $\Z_2$-graded Morita trivialization of $\A$ (given by a $\Z_2$-graded Hilbert space bundle 
$\E$ and an isomorphism $\A\to \K(\E)$ preserving gradings) is a class
\[ \on{DD}(\A)\in H^3(X,\Z)\times H^1(X,\Z_2).\]
The first component is the DD class of $\A$ after forgetting the $\Z_2$-grading. 
If it vanishes, so that $\A$ is realized as compact operators on a Hilbert space bundle $\E$,  
then the second component is the obstruction to introducing a 
compatible $\Z_2$-grading on $\E$. From now, all our DD bundles will be $\Z_2$-graded, but 
often with the trivial $\Z_2$-grading.  

\subsubsection{Example: Clifford bundles}\label{it:exclif}
%\begin{example}
Recall that the Clifford algebra of $\R^{2n}$ with its standard Euclidean metric is a $\Z_2$-graded matrix algebra
$\Cl(\R^{2n})=\End(\wedge \C^n)$. Hence, if $V\to X$ is a Euclidean vector bundle of even rank, 
the bundle $\Cl(V)$ of complex Clifford algebras, with its standard $\Z_2$-grading, is a DD bundle. 
The $H^3(X,\Z)$-component of  $\on{DD}(\Cl(V))$ is the third integral Stiefel-Whitney class of $V$, while the
$H^1(X,\Z_2)$-component measures the orientability of $V$. Given a $\Spin_c$-structure on $V$, the associated 
($\Z_2$-graded) bundle $\ca{S}\to X$ of spinors defines a Morita trivialization of $\Cl(V)$. In fact, as observed by 
Connes \cite{con:non} and Plymen \cite{ply:st} a $\Spin_c$-structure on $V$ may be \emph{defined} to be
a Morita trivialization of $\Cl(V)$, and we will take this viewpoint for the rest of this paper. 
%\end{example}

\subsubsection{Morita morphisms}\label{subsubsec:mormor}
Generalizing Morita trivializations, one has the notion of \emph{Morita morphisms}.
Suppose $\A_i\to X_i,\ i=1,2$ are two DD bundles modeled on $\K(\H_i)$.  
A Morita morphism 
\begin{equation}\label{eq:morita}
 (\Phi,\E)\colon (X_1,\A_1)\da (X_2,\A_2)
\end{equation}
is a proper map $\Phi\colon X_1\to X_2$ together with a Banach bundle $\E\to X_1$ of 
bi-modules 
\[ \Phi^*\A_2\circlearrowright
\E \circlearrowleft \A_1,\] 
locally modeled on
$\K(\H_2)\circlearrowright \K(\H_1,\H_2)\circlearrowleft \K(\H_1)$. 
The existence of such Morita morphism is equivalent to $\on{DD}(\A_1)=\Phi^*\on{DD}(\A_2)$. 
%
%In terms of the associated principal bundles, a Morita morphism amounts to 
%a lift of the structure group of the bundle $\Phi^*\A_2\otimes\A_1^{\on{op}}$
%from $\on{PU}(\H_2)\times
%\on{PU}(\H_1^{\on{op}})$  to
%$\on{PU}(\H_2\otimes\H_1^{\on{op}})$. 
The composition of Morita morphisms 
$(\Phi,\E)\colon (X_1,\A_1)\da (X_2,\A_2)$ and 
$(\Phi',\E')\colon (X_2,\A_2)\da (X_3,\A_3)$ 
is given by 
\[ (\Phi',\E')\circ (\Phi,\E)=(\Phi'\circ \Phi,\ \Phi^*\E'\otimes_{\Phi^*\A_2}\E)\colon (X_1,\A_1)\da (X_3,\A_3)\]
using the fiberwise completion of the algebraic tensor product over $\Phi^*\A_2$. 
The `identity' morphism for $\A$ is given by $(\on{id},\A)\colon (X,\A)\da (X,\A)$.  

For any Morita morphism \eqref{eq:morita} there is also a \emph{conjugate} Morita morphism 
$(\Phi,\ol{\E})\colon (X_1,\A_1^{\on{op}})\da  (X_2,\A_2^{\on{op}})$. 
Here $\ol{\E}$ is equal to $\E$ as a real vector 
bundle, but with the conjugate scalar multiplication. Denoting by $v^*\in \ol{\E}$ the element 
corresponding to $v\in \E$, the bimodule structure on $\ol{\E}$ is given by 
\[ a_1.v^*.a_2=(a_2^*. v. a_1^*)^*,\ \ \ \]
for sections $v,a_1,a_2$ of $\E,\ \A_1^{\on{op}},\ \Phi^*\A_2^{\on{op}}$.  
If $\Phi$ is invertible, we may view $(\Phinv)^*\ol{\E}$ as a $(\Phinv)^*\A_1-\A_2$ bimodule. 
The morphism $(\Phinv,(\Phinv)^*\ol{\E} )$ is then inverse to $(\Phi,\E)$.
% in the sense that 
%\[ (\Phi,\E)\circ (\Phinv,(\Phinv)^*\ol{\E} )=(\on{id},\A_2),\ \ 
%(\Phinv,(\Phinv)^*\ol{\E} )\circ (\Phi,\E)=(\on{id},\A_1).\]

\subsubsection{Twistings by line bundles}
Given DD bundles $\A_i\to X_i,\ i=1,2$ and a map $\Phi\colon X_1\to X_2$, the set of Morita morphisms
\eqref{eq:morita} is either empty, or is a principal homogeneous space under the group of 
$\Z_2$-graded line bundles $L\to X$. That is, any two Morita $\Phi^*\A_2-\A_1$ bimodules $\E,\E'$ are related 
by 
\[ \E'=\E\otimes L;\ \ \  L=\on{Hom}_{\Phi^*\A_2-\A_1}(\E,\E')\]
with $L$ the line bundle given by the bimodule homomorphisms.  A \emph{2-isomorphism} 
\[ (\Phi,\E)\simeq (\Phi,\E')\] 
between two Morita morphisms is a trivialization of this line bundle. For example, any two $\Spin_c$-structures on a Euclidean vector bundle 
$V$ (cf.~ Example \ref{it:exclif})
are related by a $\Z_2$-graded
line bundle; two $\Spin_c$-structures on $V$  are isomorphic if this line bundle is trivializable.

\subsubsection{Relative DD bundles}\label{subsubsec:rel}
Let $\Psi\colon X\to Y$ be a continuous map. By a \emph{relative DD bundle} for the given map $\Psi$, we mean a 
DD bundle $\A\to Y$ over the target together with a Morita trivialization $(\Psi,\E)\colon (X,\C)\da (Y,\A)$. 
Suppose $(\E_i,\A_i)$ are relative DD bundles for given map $\Psi_i\colon X_i\to Y_i$. 
A Morita morphism from $(\E_1,\A_1)$ to $(\E_2,\A_2)$ is given by a pair of maps 
$\tau_X\colon X_1\to  X_2$ and $\tau_Y\colon Y_1\to Y_2$ and a Morita morphism 
$(\tau_Y,\F)\colon (X_1,\A_1)\da (X_2,\A_2)$ such that 
\begin{equation}\label{eq:cop}
 (\tau_Y,\F)\circ (\Psi_1,\E_1)\simeq (\Psi_2,\E_2)\circ (\tau_X,\C).
 \end{equation}
Morita isomorphism classes of such pairs for a given map $\Psi\colon X\to Y$ are classified by 
the relative cohomology group $H^3(\Phi,\Z)$ in the ungraded case, and by $H^3(\Phi,\Z)\times H^1(\Phi,\Z_2)$ 
in the $\Z_2$-graded case. (See Appendix \ref{app:rel} for a brief discussion of relative cohomology 
groups.) We denote by $\on{DD}(\E,\A)\in H^3(\Phi,\Z)\times H^1(\Phi,\Z_2)$ the class of the relative bundle $(\E,\A)$. Given a morphism as in \eqref{eq:cop}, one has 
\[\on{DD}(\E_1,\A_1)=\tau^* \on{DD}(\E_2,\A_2),\] 
using the map in relative cohomology defined by $\tau=(\tau_X,\tau_Y)$. Furthermore, 
$\on{DD}(\ol{\E},\A^{\on{op}})=-\on{DD}(\E,\A)$. Note that if $L\to Y$ is a line bundle, then $\on{DD}(\E,\A)=\on{DD}(\E\otimes\Phi^*L,\A)$. 

\subsubsection{$G$-equivariant DD bundles}
Let $G$ be a compact Lie group. The theory outlines above generalizes to the $G$-equivariant 
case, with straightforward modifications. $G$-equivariant $\Z_2$-graded Dixmier-Douady bundles $\A\to X$ 
are classified by equivariant Dixmier-Douady classes 
$\on{DD}_G(\A)\in H^3_G(X,\Z)\times H^1_G(X,\Z_2)$. Note that if $G$ is connected, then 
$H^1_G(X,\Z_2)=H^1(X,\Z_2)$.

\subsection{Twisted $K$-homology}
In their classical paper, Donovan and Karoubi \cite{don:gr} defined twistings of the K-theory 
of a space $X$ by torsion classes in $H^3(X,\Z)$, represented by Azumaya bundles.  
J. Rosenberg \cite{ros:co} generalized to non-torsion classes, by defining the 
twisted K-theory of a space $X$ with DD bundle $\A\to X$ to be the K-theory of the $C^*$-algebra of sections of $\A$. A similar definition gives the twisted K-homology groups. 
We will follow Kasparov's approach \cite{ka:op,ka:eq} towards K-homology of $C^*$-algebras, as described in Higson-Roe's 
book \cite{hig:ana}. We will describe the main features of twisted equivariant K-homology needed for our quantization procedure; some details are deferred to Appendix \ref{app:khom}. 

Let $\A\to X$ be a $\Z_2$-graded $G$-DD bundle. Let $\Gamma_0(X,\A)$ denote the $\Z_2$-graded $G$-$C^*$-algebra of 
continuous sections of $\A$ vanishing at infinity (i.e.~ the closure of the sections of compact 
support). We define the $\A$-twisted equivariant $K$-homology of $X$ to be the $K$-homology of this 
$G$-$C^*$-algebra:
\[ K_\bullet^G(X,\A)=K^\bullet_G(\Gamma_0(X,\A)).\]
$K$-homology is a covariant functor relative to ($\Z_2$-graded, $G$-equivariant) Morita morphisms for which the underlying map $\Phi$ is proper. That is, any such Morita morphism  
$(\Phi,\E)\colon (X_1,\A_1)\da (X_2,\A_2)$ induces a group homomorphism (cf.~ Appendix \ref{app:khom})
\[ K_\bullet^G(\Phi,\E)\colon K_\bullet^G(X_1,\A_1)\to K_\bullet^G(X_2,\A_2).\]
%
%Compositions of Morita morphisms go to compositions of group homomorphisms, and the identity Morita morphism 
%$(\on{id},\A)$ of $(X,\A)$ induces the identity on $K_\bullet^G(X,\A)$. 
The map $K_\bullet^G(\Phi,\E)$ depends only on the 2-isomorphism class of $(\Phi,\E)$. 

\begin{remark}
In the definition of the groups $K_\bullet^G(X,\A)$, 
it is usually necessary to keep track of the $G$-DD bundle $\A$ as a `base point', rather than just 
record its DD class. Indeed, Morita automorphisms of $\A$ may act non-trivially on the 
$K$-groups. The group of $G$-equivariant Morita automorphisms of $\A$ is identified with the group of  
$G$-equivariant $\Z_2$-graded line bundles, by associating to any such line bundle $L\to X$
the Morita automorphism $(\id,\A\otimes L)$. The resulting 
automorphism of $K_\bullet^G(X,\A)$ depends only on the isomorphism class of 
$L$. That is, we obtain an action of $H^2_G(X,\Z)\times H^0_G(X,\Z_2)$ on 
$K_\bullet^G(X,\A)$. 

If $H^2_G(X,\Z)=0$, then all $G$-equivariant line bundles over $X$ are equivariantly isomorphic to the trivial 
line bundle. If furthermore $H^1(X,\Z_2)=0$, then the twisted $K$-groups for 
trivially graded $G$-DD bundles with prescribed class $\on{DD}_G(\A)$ are 
canonically isomorphic. (The situation is similar to the dependence of homotopy groups on a 
chosen base point.) 
\end{remark}
 
An important example of a twisted $K$-homology class is the following.  
\begin{example}
Recall that any compact manifold $M$ (possibly non-oriented)
has a fundamental class in $H_{\dim M}(M,\on{o}_{TM})$, the top degree homology group 
with coefficients in the orientation bundle. An orientation amounts to a trivialization of $\on{o}_{TM}$, 
and identifies this twisted homology group with an 
ordinary homology group. (If $M$ is non-compact, one has a similar fundamental class in Borel-Moore 
homology.)  In $K$-theory, the role of an orientation bundle is played by the Clifford bundle. 
Suppose $M$ is a Riemannian $G$-manifold of even dimension. There is a distinguished 
\emph{$K$-homology fundamental class} 
\[ [M]\in K_0^G(M,\Cl(TM)).\]
In Appendix \ref{app:khom}, we will recall Kasparov's explicit construction of this class 
\cite[$\mathsection$ 4]{ka:eq}. The isomorphism 
$K_0^G(M,\Cl(TM))\to K_0^G(M)$ for a given $\Spin_c$-structure takes $[M]$ 
to the class of the $\Spin_c$-Dirac operator (see Proposition \ref{prop:roe}).  
\end{example}

For any $G$-DD bundle $\A\to X$, there is a group isomorphism
\begin{equation}\label{eq:duality} *\colon K_\bullet^G(X,\A)\to K_\bullet^G(X,\A^{\on{op}})\end{equation}
with the property $*^2=\on{id}$ (see Appendix \ref{app:khom}). For $X=\pt,\ \A=\C$ we have $\A^{\on{op}}=\A$, and this involution is just the standard involution on the 
representation ring $R(G)=K_0^G(\pt,\C)$, given by complex conjugation of characters.

\subsection{Localization}
\label{subsec:local}
The Atiyah-Segal localization theorem \cite{at:2} carries over to the twisted case. See Freed-Hopkins-Teleman 
\cite{fr:twi} for a discussion of localization in twisted equivariant $K$-theory; we will use instead a $K$-homological formulation.
To state the Theorem suppose that $t\in Z(G)$
lies in the center of $G$. For any $R(G)$-module $X$, let $X_{(t)}$ indicate the localized module 
relative to the ideal of characters vanishing at $t$. Elements of $X_{(t)}$ may be written as 
fractions, with enumerator in $X$ and denominator a character $\chi\in R(G)$ with $\chi(t)\not=0$. 
\begin{theorem}[Localization]\label{th:locali}
Let $M$ be a compact $G$-manifold with a $G$-DD bundle $\A\to M$.
For any $t\in Z(G)$, the map 
\[ \iota_*\colon K_\bullet^G(M^t,\A|_{M^t})\to K_\bullet^G(M,\A)\]
given by the inclusion of the fixed point set $M^t$ becomes an isomorphism after localization at $t$. 
% Similarly, (assuming for simplicity that $M^t$ has even codimension)
% the map 
% \[ \iota^!\colon K_i^G(M,\A)\to K_i^G(M^t,\A|_{M^t}\otimes \Cl(\nu)),\]
% given by restriction to a tubular neighborhood of $M^t$ followed by the Thom isomorphism,
% becomes an isomorphism after localization at $t$. (Here $\nu=TM|_{M^t}/TM^t$ is the normal bundle.)  
% The inverse map $(i_*)^{-1}$ factors through $i^!$. 
\end{theorem}

The proof is a straightforward modification of the argument in \cite{at:2}: One first argues that 
if $H$ is a closed subgroup of $G$ with $t\not\in H$, then $K_\bullet^G(G/H,\A)_{(t)}=0$ for 
any $G$-DD bundle $\A\to G/H$. (In this case, the $R(G)=K_0^G(\pt)$-module structure comes from an $R(H)=K^0_G(G/H)$-module 
structure. But $R(H)_{(t)}=0$.)  
The same is then true for closed tubular neighborhoods 
of orbits in $M-M^t$. The general result follows by induction, using a Mayer-Vietoris argument.

\begin{remark}
If $U$ is a $G$-invariant open subset of $M$ with $M^t\subset U$, then the localized restriction map 
$K_\bullet^G(M,\A)_{(t)}\to K_\bullet^G(U,\A|_U)_{(t)}$ is an isomorphism. This follows by excision, since $K_\bullet^G(M\backslash U,\A|_{M\backslash U})_{(t)}=0$. Hence, 
the inverse map $(i_*)^{-1}\colon K_\bullet^G(M,\A)_{(t)}\to K_\bullet^G(M^t,\A|_{M^t})_{(t)}$ factors through the restriction to any open neighborhood of $M^t$. 
\end{remark}

%The isomorphism \eqref{eq:KK} translates $\iota_*$ into the wrong-way functoriality map
% 
%\[ \iota_! \colon K^i_{G,cp}(TM^t,\pi_{M^t}^*\A^{\on{op}}|_{M^t})\to K^i_{G,cp}(TM,\pi_M^*\A^{\on{op}}).\]
%Its composition $\iota^*\circ \iota_!$ with the pull-back 
% 
%\[ \iota^*\colon K^i_{G,cp}(TM,\pi_M^*\A^{\on{op}})\to K^i_{G,cp}(TM^t,\pi_{M^t}^*\A^{\on{op}}|_{M^t})\]
% 
%is multiplication by the pull-back of $\lambda_{-1}(\nu\otimes \C)\in K^0_{G,cp}(M^t)$, where 
%$\nu=TM|_{M^t}/TM^t\to M^t$ is the normal bundle. Atiyah-Segal argue that $ \lambda_{-1}(\nu\otimes \C)$ is invertible %after localization at $t$, so that $\lambda_{-1}(\nu\otimes \C)^{-1}\circ\iota^*$ is the inverse map to 
%$\iota_!$. 

\subsection{$K$-theoretic interpretation of the fusion ring}\label{subsec:mult}
We begin by introducing the notion of a multiplicative DD bundle, similar to that of a multiplicative gerbe, 
as studied e.g. in Brylinski-McLaughlin \cite{br:ge1} and Waldorf \cite{wal:mul}. 
Let $G$ be a compact Lie group, viewed as a $G$-space by the conjugation action. We denote by 
\[ \on{Mult}\colon G\times G\to G,\ \ \on{Inv}\colon G\to G,\ \ i\colon \pt \to G\]
the group multiplication, the inversion and the inclusion of the group unit. 

\begin{definition}
A \emph{multiplicative DD bundle over $G$} is a $G$-DD bundle $\A\to G$ together with a 
a $G$-equivariant Morita morphism 
\begin{equation}\label{eq:multi}
 (\on{Mult},\F_{\on{Mult}})\colon (G\times G,\A\times\A)\da (G,\A),
 \end{equation}
such that the following associativity property holds:
\[ (\on{Mult},\F_{\on{Mult}})\circ (\on{Mult}\times\id ,\F_{\on{Mult}}\times \A)\simeq (\on{Mult},\F_{\on{Mult}})\circ (\id \times \on{Mult} ,\A\times \F_{\on{Mult}}).\]
\end{definition}
From the multiplication morphism $\F_{\on{Mult}}$, one obtains $G$-equivariant 
Morita morphisms 
\begin{equation}\label{eq:incinv}  \begin{split}(i,\F_i)&\colon (\pt,\C)\da (G,\A)
\\ (\on{Inv},\F_{\on{Inv}})&\colon (G,\A^{\on{op}})\da (G,\A).
\end{split}\end{equation}
acting as a unit and inversion, respectively. (For example, to construct $\F_i$, note that $\F_{\Mult}$ restricts to a Morita morphism 
$\A_e\otimes\A_e\da \A_e$. Multiply by $\A_e^{\on{op}}$, and use that $\A_e\otimes \A_e^{\on{op}}$ is canonically Morita trivial, to get a Morita isomorphism $\A_e\da \C$.) 

For a multiplicative $G$-DD bundle $\A\to G$, the twisted $K$-homology group 
\begin{equation}\label{eq:r} \mathsf{R}_\bullet=K_\bullet ^G(G,\A)\end{equation}
acquires a ring structure, with product $\mathsf{R}_\bullet\otimes\mathsf{R}_\bullet\to \mathsf{R}_\bullet$ 
given as push-forward $\on{Mult}_*=K_\bullet^G(\on{Mult},\F_{\on{Mult}})$. Furthermore, 
there is an $R(G)$-module homomorphism 
\[ R(G)\to \mathsf{R}_\bullet\]
given as push-forward $i_*=K_\bullet^G(i,\F_i)$, and an involution 
\[ *\colon \mathsf{R}_\bullet\to \mathsf{R}_\bullet\]
obtained as a composition of the isomorphism $K_\bullet^G(G,\A)\to K_\bullet^G(G,\A^{\on{op}})$ from \eqref{eq:duality}
with the push-forward $\on{Inv}_*=K_\bullet^G(\on{Inv},\F_{\on{Inv}})$. 
 
The ring $\mathsf{R}_\bullet$ was computed by Freed-Hopkins-Teleman \cite{fr:lo1} in great generality. 
We will only consider the case that the group $G$ is simply connected. Let $G=G_1\times\cdots\times\cdots G_N$ be the decomposition into simple factors. The 
cohomology of $G$ vanish in degrees 1 and 2 (both equivariantly and non-equivariantly), while 
$H^3_G(G,\Z)=H^3(G,\Z)=\Z^l$ has no torsion. The classes $x\in H^3_G(G,\Z)$ are primitive, that is 
$\on{Mult}^* x=\pr_1^* x+\pr_2^* x$. As a consequence, any $G$-DD bundle 
$\A\to G$ has a multiplicative structure, where $\F_{\on{Mult}}$ is unique up to 2-isomorphism. 
We say that $\A$ is \emph{at level $l$} if $\on{DD}_G(\A)=l$ under the isomorphism 
$H^3_G(G,\Z)=\Z^l$. 

For any $k\in \Z^N$, $k_i\ge 0$ we denote by $R_k(G)$ the \emph{level $k$ fusion ideal} (Verlinde algebra). It is a quotient of the representation ring by the \emph{level $k$ fusion ideal} $I_k(G)\subset R(G)$ (see Appendix \ref{app:fus}).  
\begin{theorem}[Freed-Hopkins-Teleman \cite{fr:lo1}] \label{th:fht}
  Let $k\in \Z^N,\ k_i\ge 0$ be a given level, and let $\A\to G$ be a $G$-DD bundle at the shifted level 
  $k+\cox$, where $\cox=(\cox_1,\ldots,\cox_N)$ are the dual Coxeter numbers. Then  
  \[ K_\bullet ^G(G,\A) \cong R_k(G)\]
as a $\Z_2$-graded ring with involution (in particular $K_1^G(G,\A)=0$). 
  More precisely, the homomorphism $R(G)\to K_\bullet ^G(G,\A)$ is onto, with kernel
  $I_k(G)$.  
\end{theorem}

\section{Pre-quantization of group-valued moment maps}
Let $G$ be a Lie group, with Lie algebra $\g$, and suppose $B$ is a
$G$-invariant, symmetric, nondegenerate bilinear form on $\g$.
Equivalently, $G$ carries a bi-invariant Riemannian metric.
The left, right invariant Maurer-Cartan-forms on $G$ will be denoted 
$\theta^L,\theta^R\in\Om^1(G,\g)$. We denote by $\eta\in\Om^3(G)$ the Cartan 3-form, 
\[ \eta=\f{1}{12}B(\theta^L,[\theta^L,\theta^L]).\]
The form $\eta$ is bi-invariant, hence closed. 
It has a closed equivariant (relative to the
conjugation action of $G$) extension $\eta_G\in \Om_G^3(G)$ in the Cartan model of equivariant cohomology 
(cf.~ Appendix \ref{app:eq}),
\[ \eta_G(\xi)=\eta-\f{1}{2}B(\theta^L+\theta^R,\xi).\]
%

%%%%%%%%%%%%%%%%%%%%%%%%%%%%%%%%%%%%%%%%%%%%%%%%%%%%%%%%%%%%%%%%%%%%%%%%%%%%%%%%%

\subsection{Group-valued moment maps}\label{subsec:gro}
The following  concept of a $G$-valued moment map was introduced in \cite{al:mom}.
\subsubsection{Definition}
A \emph{q-Hamiltonian $G$-space} is a triple
$(M,\omega,\Phi)$ where $M$ is a $G$-manifold, $\omega\in \Om^2(M)$ is
an invariant 2-form, and $\Phi\in C^\infty(M,G)$ is a $G$-equivariant
map such that the following two axioms hold:
\begin{enumerate}
\item[(i)] $\d_G\omega=-\Phi^*\eta_G$, 
\item[(ii)] $\Phi$ is transversal to the kernel of $\omega$,  i.e.~  
\[\ker(\om_m)\cap \ker(\d_m\Phi)=\{0\},\ \ \ m\in M.\] 
\end{enumerate}
  Assuming (i), condition (ii) is equivalent to the 
  assertion that the kernel of $\omega_m$ equals the image of $\ker(\Ad_{\Phi(m)}+I)\subset\g$ 
  under the infinitesimal action $\g\to T_mM$. This equivalence was noted independently by
  Bursztyn-Crainic \cite{bur:di} and Xu \cite{xu:mom}.
\subsubsection{Examples}
We recall the basic examples and constructions with group-valued moment maps.
We will not be very specific about the
2-forms, but mainly describe the action and the moment map. See
\cite{al:mom,al:pur} or the indicated references for further details.
\begin{enumerate}
\item {\bf Conjugacy classes}.  Let $\Co$ be any conjugacy class, and
      denote by $\Phi\colon \Co\hra G$ the inclusion. $\Co$ carries a
      unique 2-form $\om$ for which $(\Co,\om,\Phi)$ is a
      q-Hamiltonian $G$-space. It is the q-Hamiltonian counterpart to coadjoint orbits $\O\subset\g^*$
      as Hamiltonian $G$-spaces.  
\item {\bf The double}. Let $\mathbf{D}(G)=G\times G$ with $G\times
      G$-action $(g_1,g_2).(a,b)=(g_1 a g_2^{-1},\ g_2 b g_1^{-1})$.
      Put $\Phi(a,b)=(ab,a^{-1}b^{-1})$. Then $\mathbf{D}(G)$ is a
      q-Hamiltonian $G\times G$-space. It is the q-Hamiltonian
      counterpart to $T^*G$ as a Hamiltonian $G\times G$-space.
      Let
      $D(G)$ be equal to $\mathbf{D}(G)$ as a manifold, but with the
      diagonal action $g.(a,b)=(gag^{-1},gbg^{-1})$. Then $D(G)$ is a
      q-Hamiltonian $G$-space with moment map
      $\Phi(a,b)=aba^{-1}b^{-1}$.
\item {\bf Spheres, quaternionic projective spaces}.  Suppose
      $G=\SU(n)$. The action of $G$ on $\C^n\oplus \R\cong \R^{2n+1}$
      restricts to an action on the unit sphere $M=S^{2n}$.  It turns
      out that there exists a 2-form and moment map for this action,
      making $(S^{2n},\om,\Phi)$ into a q-Hamiltonian $\SU(n)$-space.
      This example was found in \cite{al:du} for $n=2$, and \cite{hu:imp}
      for $n>2$. Recently, A. Eshmatov \cite{esh:ex} found that similarly, the quaternionic
      projective space $\mathbb{H} P(n)$ is a q-Hamiltonian
      $\on{Sp}(n)$-space. 
\item {\bf Moduli spaces of flat $G$-bundles.}
Let $\Sig_h^r$ be a compact connected oriented surface of genus $h$ with $r\ge 1$ boundary components. 
Fix a base point $x_i$ on each boundary component, and let $\M_G(\Sig_h^r;x_1,\ldots,x_r)$ be the moduli space of flat $G$-bundles on $\Sig_h^r$, up to gauge transformations that are trivial at the base points. The full gauge group 
action descends to an action of $G^r$ on the moduli space, and makes 
$\M_G(\Sig_h^r;x_1,\ldots,x_r)$ into a q-Hamiltonian $G^r$-space, with moment map
given by the holonomy around the boundary components. Special cases are $\M_G(\Sig_0^2;x_1,x_2)\cong \mathbf{D}(G)$ and $\M_G(\Sig_1^1,x)\cong D(G)$, where the isomorphism depends on a choice of generators for the 
fundamental groupoid of $\Sig_h^r$. 
% The group $\mathcal{G}(\Sig)$ acts on $\Om^1(\Sig,\g)$ by gauge transformations, 
% $g.A=\Ad_g(A)-g^*\theta^R$. Fix a base point $x_i$ on each boundary component, and let 
% $\mathcal{G}(\Sig;x_1,\ldots,x_r)\subset \mathcal{G}(\Sig)$ be the subgroup of 
% gauge transformations such that $g(x_i)=e$ for all $i=1,\ldots,r$. The space 
%
% \[ \M_G(\Sig;x_1,\ldots,x_r)=\{A\in \Om^1(\Sig,\g)|\ \d A+\hh[A,A]=0\}\Big/\mathcal{G}(\Sig;x_1,\ldots,x_r)\]
%
% of flat connections up to based gauge equivalence is a q-Hamiltonian 
% $G^r$-space. The action of $G^r$ is the action of the quotient group 
% $\mathcal{G}(\Sig)/\mathcal{G}(\Sig;x_1,\ldots,x_r)$, and the moment map is 
% defined by the holonomies around the boundary 
% circles. Letting $\Sig_h^r$ denote the connected oriented surface of genus $h$ with $r$ boundary components, 
% one finds that $\M_G(\Sig_0^2;x_1,x_2)\cong \mathbf{D}(G)$.
%
\end{enumerate}

\subsubsection{Conjugates} 
Let $(M,\om,\Phi)$ be a q-Hamiltonian $G$-space. Denote by $M^*$ the space $M$ with the same $G$-action, 
2-form $\om^*=-\om$, and moment map $\Phi^*=\on{Inv}\circ \Phi$ where $\on{Inv}\colon G\to G,\ g\mapsto g^{-1}$
is the inversion map.  Then $(M^*,\om^*,\Phi^*)$ is a q-Hamiltonian
    $G$-space called the \emph{conjugate} of $M$. Clearly $(M^*)^*=M$. 
\subsubsection {Fusion} Let $(M,\om,\Phi)$ be a q-Hamiltonian $G\times G$-space.
Denote by $M_{\on{fus}}$ the space $M$ with the diagonal $G$-action, let $\Phi_{\on{fus}}\colon M_{\on{fus}}\to G$ be the map obtained by composing $\Phi$ with group 
multiplication, and put $\om_{\on{fus}}=\om+\Phi^*\sig$ where
\[ \sig=\hh B(\pr_1^*\theta^{L},\pr_2^*\theta^{R})\in\Om^2(G\times G).\]
Then $(M_{\on{fus}},\om_{\on{fus}},\Phi_{\on{fus}})$ is a q-Hamiltonian
$G$-manifold, called the \emph{fusion} of $M$.  In case
$M=M_1\times M_2$ is the direct product of q-Hamiltonian $G$-manifolds, we denote $M_{\on{fus}}$ by $M_1\fus M_2$, or simply
$M_1\times M_2$ if there is no risk of confusion. As an example, $\mathbf{D}(G)_{\on{fus}}=D(G)$. 
% More
%generally, $M(\Sig_h^r)\cong G^{2(h+r-1)}$ is isomorphic to the space obtained by fusing $r-1$
% copies of $\mathbf{D}(G)$ with $h$ copies of $D(G)$.
%
\subsubsection{Reduction}
Suppose $(M,\om,\Phi)$ is a q-Hamiltonian $G$-space, and that 
$e$ is a regular value of $\Phi$. Then $G$ acts locally freely on $Z=\Phi^{-1}(e)$, and the 
\emph{symplectic
      quotient} $M_\red\equiv M\qu G=Z/G$ is a  
symplectic orbifold, with 2-form $\om_\red$ induced from $\om$. If the $G$-action on $\Phinv(e)$ 
is free, then $M\qu G$ is a smooth symplectic manifold.  On the other hand, if $e$ is a singular value then it is a
      stratified symplectic space in the sense of Sjamaar-Lerman \cite{sj:st}.

By a result of \cite{al:mom}, the moduli space of flat $G$-bundles over $\Sig_h^r$, 
with boundary holonomies in prescribed conjugacy classes $\Co_j\subset G$, is a symplectic quotient
\[ M(\Sig_h^r;\Co_1,\ldots,\Co_r)=(D(G)^h\fus \Co_1\cdots\fus \Co_r)\qu G.\]
Here the symplectic structure coincides with the standard one, constructed by Atiyah-Bott \cite{at:ge,at:mo} using gauge theory. The case of a surface without boundary ($r=0$) is
included as $M(\Sig_h^0)=M(\Sig_h^1)\qu G$.
\subsubsection{Exponentials}\label{subsubsec:exp} 
Let $\varpi\in\Om^2(\g)$ denote the
      primitive of $-\exp^*\eta\in\Om^3(\g)$ given by the de Rham
      homotopy operator.  Suppose $(M_0,\om_0,\Phi_0)$ is an ordinary
      Hamiltonian $G$-space, with moment map $\Phi_0\colon M_0\to
      \g^*\cong \g$. Let $M=M_0$ with the same $G$-action, moment map
      $\Phi=\exp(\Phi_0)$, and 2-form $\om=\om_0+\Phi_0^*\varpi$.
      Then $(M,\om,\Phi)$ satisfies condition (i) of a q-Hamiltonian
      $G$-manifold, while condition (ii) holds over the open subset of
      all $m\in M$ such that $\d\exp|_{\Phi_0(m)}$ has maximal rank.
\subsubsection{Morphisms}      
A \emph{morphism} $\varphi$ from a q-Hamiltonian $G$-space $(M,\om,\Phi)$ to a
q-Hamiltonian $G'$-space $(M',\om',\Phi')$ is given by a group
homomorphism $\varphi_G\colon G\to G'$ such that the map of Lie algebras 
preserves bilinear forms, together with
a smooth map $\varphi_M\colon M\to M'$ satisfying
\begin{equation}\label{eq:morphism}
 \varphi_M^*\om'=\om,\ \ \Phi'\circ \varphi_M=\varphi_G\circ \Phi,\ \ \varphi_M(g.x)=\varphi_G(g).\varphi_M(x).
 \end{equation}
An invertible morphism from $(M,\om,\Phi)$ to itself is called an automorphism. 
\begin{proposition}\label{prop:fix}
  Suppose a compact Lie group $K$ (possibly disconnected) acts by
  automorphisms of the q-Hamiltonian $G$-space $(M,\om,\Phi)$.  Then
  the fixed point set $M^K$ is a q-Hamiltonian $G^K$-space, in such a
  way that the inclusion $M^K\hra M$ is a
  q-Hamiltonian morphism.
\end{proposition}
\begin{proof}
Property (i)
      is immediate. For (ii) suppose $m\in M^K$. For $k\in K$ let $k_m$ denote the linearized 
      action on $T_mM$. The invariant subspace $T_mM^K\subset T_mM$ has a unique
      $K$-invariant complement, spanned by the images of $(\on{id}-k_m)$ for $k\in K$.  
      Since $\om_m$ is $K$-invariant, the resulting splitting of $T_mM$ is
      $\om_m$-orthogonal. It follows that $\ker(\om^K)_m=\ker(\om_m)^K$. Hence 
\[ \ker(\om^K)_m\cap \ker(\d\Phi^K)_m=\ker(\om_m)^K\cap \ker(\d_m\Phi)^K=
(\ker(\om_m)\cap \ker(\d_m\Phi))^K=0.\qedhere\]
\end{proof}
As a special case, if $(M,\om,\Phi)$ is a q-Hamiltonian $G$-space, and $a\in G$, 
then the components of the fixed point set $M^a$ are q-Hamiltonian
$G^a$-spaces, in such a way that the inclusion map is a morphism.
Some other examples of morphisms:

\begin{enumerate}
\item (Finite quotients.)
 Suppose a finite group $\Gamma$ acts on $(M,\om,\Phi)$, preserving
  the 2-form and moment map. Then $M/\Gamma$ inherits the structure of
  a q-Hamiltonian $G$-manifold, in such a way that the quotient map is
  a morphism.  For instance, if $\Gamma$ is a subgroup of the center of 
  $G$, the quotient $D(G)/(\Gamma\times\Gamma)=D(G/\Gamma)$ is a q-Hamiltonian 
$G$-space.
\item (Coverings.)
    Suppose $(M,\om,\Phi)$ is a q-Hamiltonian $G$-space, and $\wt{G}$ a
  finite cover of $G$. Let $\wt{M}\subset M\times \wt{G}$ be
  the fiber product with respect to $\Phi\colon M\to G$ and the
  projection $\wt{G}\to G$. Then $\wt{M}$ is a Hamiltonian $\wt{G}$-space in such a way that 
  the quotient map is a morphism. 
\item (Cross-sections.)
  Suppose $H\subset G$ is a closed subgroup, and $V\subset G$ is an
  $H$-equivariant \emph{cross-section}, i.e.~ an $H$-invariant
  submanifold such that the map $G\times_H V\to G$ is a diffeomorphism
  onto an open subset of $G$.  Then the pre-image $Y=\Phinv(V)$
  carries the structure of a q-Hamiltonian $H$-space, in such a way that the 
  inclusion is a q-Hamiltonian morphism.
\end{enumerate}

%\footnote{The space $M(\Sig_0^2)$ carries an involution, $F(u,v)=
%  (v,v^{-1}).(u,v)=(\Ad_v(u),v^{-1})$. This involution is equivariant,
%  relative to the automorphism of $G\times G$ interchanging the two
%  factors, and $(F^*\Phi)(u,v)=(\Ad_v(u),u^{-1})$. The 2-form
%  transforms as $F^*\om=-\om$.  } The space $M(\Sig_1^1)\cong G^2$ is
% obtained from this fundamental example by fusion: The action is the
% diagonal action, and the new moment map is $(u,v)\mapsto
% uvu^{-1}v^{-1}$. (In fact, $M(\Sig_1^1)\cong D(G)$.)  More generally,
% any $M(\Sig_h^r)\cong G^{2h+2(r-1)}$ with $r\ge 1$ is obtained by
% `fusion' (see Section \ref{subsec:fusion} below) of $h$ copies of
% $M(\Sig_1^1)$ and $r-1$ copies of $M(\Sig_0^2)$; thus
% $M(\Sig_h^r)\cong (G^2)^{h+r-1}$.
%

%%%%%%%%%%%%%%%%%%%%%%%%%%%%%%%%%%%%%%%%%%%%%%%%%%%%%%%%%%%%%%%%%%%%%%%%%%%%%%%%%

\subsection{Prequantization}
The pre-quantization of a symplectic manifold $(M,\om)$ is traditionally defined to be a line bundle 
$L\to M$ whose first Chern class is an integral lift of the cohomology class $[\om]\in H^2(M,\R)$ 
of the symplectic form. We will give a similar definition for q-Hamiltonian $G$-spaces, using 
DD bundles.    
Other viewpoints can be found in the papers of Shahbazi \cite{sha:pre} (using gerbes) and 
Laurent-Gengoux-Xu \cite{lau:qu} (using quasi-presymplectic groupoids). 

The conditions $\d_G\om=-\Phi^*\eta_G$ and $\d_G\eta_G=0$ mean that the pair $(\om,\eta_G)$ 
defines a cocycle for the relative equivariant de Rham theory (cf.~ Appendix \ref{app:rel}). 
Let $[(\om,\eta_G)]\in H^3_G(\Phi,\R)$ 
denote its cohomology class. The relative cohomology group with integer coefficients $H^3_G(\Phi,\Z)$, on the other hand, 
is realized as equivalence classes of relative $G$-DD bundles $(\E,\A)$ (cf.~ Section \ref{subsubsec:rel}), i.e.~ 
as Morita morphisms
 \begin{equation}\label{eq:pre} (\Phi,\E)\colon (M,\C)\da (G,\A). \end{equation}
\begin{definition}\label{def:preq}
Suppose $G$ is connected and semi-simple. A \emph{pre-quantization} of a q-Hamiltonian $G$-space $(M,\om,\Phi)$ is a 
relative $G$-DD bundle $(\E,\A)$ such that 
$\on{DD}_G(\E,\A)$ is an integral lift of $[(\om,\eta_G)]\in H^3_G(\Phi,\R)$. Sometimes we refer to the 
Morita morphism \eqref{eq:pre} as the pre-quantization. 
\end{definition}
In this definition, the $H^1_G(X,\Z_2)$-component of $\on{DD}_G(\E,\A)$ does not enter the condition 
of `integral lift'. In our applications, the $H^1(X,\Z_2)$-component will in fact be zero. 
\begin{remarks}
\begin{enumerate}
\item In particular, $\on{DD}_G(\A)$ must be an integral lift of $[\eta_G]$. 
Since
$\eta_G$ is defined in terms of the bilinear form $B$ on $\g$, this introduces an integrality condition on $B$.
\item 
Suppose $G=\pt$, so that the q-Hamiltonian space is just a symplectic manifold, and take $\A=\C$. A pre-quantization 
$(\Phi,\E)\colon (M,\C)\da (\pt,\C)$ is given by a $G$-equivariant pre-quantum line bundle $\E=L$. Indeed, in this case 
$H^3(\Phi,\Z)\cong H^2(M,\Z)$, and $\on{DD}_G(\E,\C)$ corresponds to the Chern class of $L$ under this isomorphism.
\item \label{it:d}
If $(\E,\A)$ is a pre-quantization, and $L\to M$ is a $G$-equivariant line bundle, then $(\E\otimes L,\A)$ is a 
pre-quantization if and only if the class $\on{DD}_G(L,\C)\in H^3_G(\Phi,\Z)$ is torsion. This class is the image of 
the equivariant Chern class $c_{1,G}(L)\in H^2_G(M,\Z)$ under the connecting homomorphism in the long exact sequence, 
\[ \cdots\to H^2_G(G,\Z) \to H^2_G(M,\Z)\to H^3_G(\Phi,\Z) \to H^3_G(G,\Z)\to H^3_G(M,\Z)\to\cdots\]
For $G$ semi-simple, the group $H^2_G(G,\Z)$ is torsion, and hence $c_{1,G}(L)$ must be torsion. That is, 
pre-quantizations with given $\A$ are unique up to flat $G$-equivariant line bundles.  
\end{enumerate}
\end{remarks}

\begin{remark}
While Definition \ref{def:preq} makes sense for \emph{any} compact $G$, it is not quite 
satisfactory if $G$ is not semi-simple. The problem becomes apparent 
if one attempts to define the pre-quantization of ordinary Hamiltonian $G$-spaces $(M,\om_0,\Phi_0)$ along similar lines. Suppose for simplicity that $H^3_G(\pt,\Z)=0$, i.e. $G$ has no non-trivial central extensions by $\U(1)$. 
Then the long exact sequence in relative cohomology shows that $H^3_G(\Phi_0,\Z)\cong H^2_G(M,\Z)/H^2_G(\pt,\Z)$. 
Thus, a class in $H^3_G(\Phi_0,\Z)$ does not quite determine a $G$-equivariant line bundle unless $G$ is semi-simple, in which case $H^2_G(\pt,\Z)=\on{Hom}(G,\U(1))$ is zero. 
\end{remark}

\begin{proposition}[Functoriality of pre-quantization]
Let $\varphi=(\varphi_M,\varphi_G)$ be a morphism from a q-Hamiltonian $G$-space $(M,\om,\Phi)$ to a q-Hamiltonian 
$G'$-space $(M',\om',\Phi')$, where $G,G'$ is again semi-simple. 
If $(\E',\A')$ is a pre-quantization of $(M',\om',\Phi')$ then $(\varphi_M^*\E',\varphi_G^*\A')$ is a 
pre-quantization of $(M,\om,\Phi)$. 
\end{proposition}
\begin{proof}
$\on{DD}_G(\varphi_G^* \E',\varphi_M^*\A')=\varphi_G^*\on{DD}_{G'}(\E',\A')$ is an integral lift of
$[(\varphi_M^*\om',\varphi_G^*\eta_G')]=\varphi^*[(\om',\eta_G')]$. 
\end{proof}

For instance, if $(M,\om,\Phi)$ is a pre-quantized q-Hamiltonian $G$-space, and $a\in G$, then 
the fixed point set $M^a$ is a pre-quantized q-Hamiltonian $G^a_{ss}$-space, where $G^a_{ss}$ is the semi-simple part (commutator subgroup) of $G^a$.   

For the remainder of this section, we will assume that our compact group $G$ is connected and semi-simple. We 
also assume that we are given a fixed \emph{multiplicative} $G$-DD bundle $\A\to G$ such that $\on{DD}_G(\A)$ is an integral lift of the class $[\eta_G]$ defined by the given bilinear form $B$ on $\g$.

\begin{proposition}[Fusion]\label{prop:fusions}
If $(\E,\A\times \A)$ is a pre-quantization of a q-Hamiltonian 
$G\times G$-space $(M,\om,\Phi)$, then 
\begin{equation}\label{eq:fussion} (\Phi_{\on{fus}},\E_{\on{fus}})=(\on{Mult},\F_{\on{Mult}})\circ (\Phi,\E)
\end{equation}
defines a pre-quantization of the fusion, $(M_{\on{fus}},\om_{\on{fus}},\Phi_{\on{fus}})$. 
\end{proposition}
\begin{proof}
Since 
\[ \on{DD}_G(\E,\A\times \A)=\on{Mult}^*\on{DD}_G(\E_{\on{fus}},\A),\ \ \ 
[(\om,\pr_1^*\eta_G+\pr_2^*\eta_G)]=\on{Mult}^*[(\om,\eta_G)]\]
 it suffices to show that the map 
\begin{equation}\label{eq:injective}
\Mult^*\colon H^3_G(\Phi_{\on{fus}},\R)\to H^3_G(\Phi,\R)\end{equation} 
is injective. 
 This map fits into a  commutative diagram, with exact rows,
\[\begin{CD}
  0 @>>> H^2_G(M,\R) @>>> H^3_G(\Phi_{\on{fus}},\R)@>>> H^3_G(G,\R) \\
  @VVV @VVV @VVV @VV{\on{Mult}^*}V \\
  0 @>>> H^2_{G}(M,\R) @>>> H^3_{G}(\Phi,\R)@>>> H^3_{G}(G\times G,\R) \\
\end{CD} 
 \]
 The right vertical map is injective;  a left-inverse is the pull-back under 
 $G\to G\times G,\ g\mapsto (g,e)$.  By a diagram chase it follows that 
\eqref{eq:injective}   is injective.
\end{proof}
\begin{proposition}[Conjugates] \label{prop:conjugates}
Suppose  $(\E,\A)$ is a pre-quantization of the q-Hamiltonian $G$-space
$(M,\om,\Phi)$. Then 
\[ (\Phi^*,\E^*):=(\on{Inv},\F_{\on{Inv}})\circ (\Phi,\ol{\E})
\colon (M,\C)\da (G,\A)
\]
defines a pre-quantization $(\E^*,\A)$ of its conjugate $(M^*,\om^*,\Phi^*)$.
\end{proposition}
\begin{proof}
The class $\on{DD}_G(\ol{\E},\A^{\on{op}})=-\on{DD}_G(\E,\A)$ is an integral lift of 
$-[(\om,\eta_G)]=[(-\om,\on{Inv}^*\eta_G)]$. Hence, by composing with 
$(\on{Inv},\F_{\on{Inv}})$ (cf.~ \eqref{eq:incinv}) we obtain an integral lift of 
$\on{Inv}^*[(-\om,\on{Inv}^*\eta_G)]=[(-\om,\eta_G)]\in H^3_G(\Phi^*,\R)$.   
\end{proof}
 
The Morita morphism $(i,\F_i)\colon (\pt,\C)\da (G,\A)$ 
extends to a $G$-equivariant Morita morphism  
\[ (\exp,\F_{\on{exp}})\colon (\g,\C)\da (G,\A)\]
lifting the exponential map. 

\begin{proposition}[Exponentials]\label{prop:exxponentials}
Suppose $(M,\om,\Phi)$ is the exponential of a Hamiltonian $G$-space $(M,\om_0,\Phi_0)$, as in Section 
\ref{subsubsec:exp}, and that $L\to M$ is a $G$-equivariant pre-quantum line bundle over $M$. Then the composition  
\begin{equation}\label{eq:exxp}
 (\Phi,\E)=(\exp,\F_{\exp})\circ (\Phi_0,L)\colon (M,\C)\da (G,\A)\end{equation}
 is a pre-quantization of $(M,\om,\Phi)$. Conversely, if $(\E,\A)$ is a pre-quantization 
 of $(M,\om,\Phi)$,  
 %so that $(\Phi_0,\E)\colon (M,\C)\da (\g,\exp^*\A)$, 
 then the line bundle $L\to M$ defined by 
\[(\Phi_0,L)=(\on{id}_\g,\ol{\F}_{\exp})\circ (\Phi_0,\E)\colon (M,\C)\da (\g,\C)\]
is a pre-quantum line bundle for $(M,\om_0,\Phi_0)$.
\end{proposition}
\begin{proof}
Since $G$ is semi-simple, we have $H^2_G(\pt,\Z)=0$, while $H^3_G(\pt,\Z)$ is the torsion group of central extensions
of $G$ by $\U(1)$. 
Let $\wt{H}^3_G(\Phi_0,\Z)$ be the kernel of the map $H^3_G(\Phi_0,\Z)\to H^3_G(\g,\Z)\cong H^3_G(\pt,\Z)$, and let
$\wt{H}^3_G(\Phi,\Z)$ denote the kernel of the map $H^3_G(\Phi,\Z)\to H^3_G(G,\Z)\to H^3_G(\pt,\Z)$. 
By the long exact sequence 
\[ \cdots\to H^2_G(\g,\Z)\to H^2_G(M,\Z)\to H^3_G(\Phi_0,\Z)\to H^3_G(\g,\Z)\to \cdots\]
and since $H^2_G(\g,\Z)=H^2_G(\pt,\Z)=0$, we have $\wt{H}^3_G(\Phi_0,\Z)\cong H^2_G(M,\Z)$.
The exponential map gives $\exp^*\colon \wt{H}^3_G(\Phi,\Z)\to \wt{H}^3_G(\Phi_0,\Z)\cong H^2_G(M,\Z)$. Together with the map to $H^3_G(G,\Z)$
this gives an inclusion 
\begin{equation}
\label{eq:phiso}
 \wt{H}^3_G(\Phi,\Z)\hra H^2_G(M,\Z)\oplus H^3_G(G,\Z).
\end{equation}
Since $\A$ is multiplicative, the $H^3$-component of its DD class $\on{DD}_G(\A)$ lies in the kernel of 
$H^3_G(G,\Z)\to H^3_G(\pt,\Z)$. Thus 
\[ \on{DD}_G(\E,\A)'\in \wt{H}^3_G(\Phi,\Z),\]
where the prime $'$ indicates the $H^3_G(\cdot,\Z)$-component 
(the $H^1_G(\cdot,\Z_2)$-component does not play a role here). 
Equation \eqref{eq:exxp} gives $\exp^*\on{DD}_G(\E,\A)'=\on{DD}_G(L,\C)'$ which is identified with the Chern class 
$c_{1_G}(L)$ under $\wt{H}^3_G(\Phi_0,\Z)\cong H^2_G(M,\Z)$. That is, 
$\on{DD}_G(\E,\A)'$ maps to $c_{1_G}(L)\oplus \on{DD}_G(\A)'$ under the inclusion
\eqref{eq:phiso}. On the other hand, the class 
$[(\om,\eta_G)]\in \wt{H}^3_G(\Phi,\R)=H^3_G(\Phi,\R)$ maps to $[\om_{0,G}]\oplus [\eta_G]$ 
under a similar isomorphism for real coefficients. It follows that $c_{1,G}(L)$ is an integral lift 
of $[\om_{0,G}]$ if and only if $\on{DD}_G(\E,\A)'$ is an integral lift of 
$[(\om,\eta_G)]$. 
\end{proof}

\begin{proposition}[Reduction]
Suppose $(\E,\A)$ is a pre-quantization of the q-Hamiltonian $G$-space $(M,\om,\Phi)$. 
Assume that $e$ is a regular value of $\Phi$ and that $G$ acts freely on $Z=\Phinv(e)$. 
Then the reduced space $(M_\red,\om_\red,\Phi_\red)$ inherits a pre-quantization. 
Indeed, defining the $G$-equivariant line bundle 
$L_Z\to Z$ by 
\[ (\Phi|_Z,L)=(\on{id}_{\pt}, \ol{\F}_i)\circ (\Phi|_Z,\E|_Z)\colon (Z,\C)\da (\pt,\C),\]
the quotient $L_\red=L_Z/G$ is a pre-quantum line bundle.  If the action of $G$ on $\Phi^{-1}(e)$ is only locally free,
one gets a similar statement with $L_{\on{red}}$ an orbifold line bundle. 
\end{proposition}
\begin{proof}
We use the notation from the proof of Proposition \ref{prop:exxponentials}. 
By a similar argument, the pull-back to $Z$ 
gives an inclusion $\wt{H}^3_G(\Phi|_Z,\Z)\subset H^2_G(Z,\Z)\oplus H^3_G(G,\Z)$, taking $\on{DD}_G(\E,\A)'$  to 
$c_{1,G}(L|_Z)\oplus \on{DD}_G(\A)'$. A similar map with real coefficients takes 
$[(\om,\eta_G)]$ to $[(\om|_Z)]\oplus [\eta_G]$. Hence $c_{1,G}(L|_Z)$ is an integral lift of 
$[\om|_Z]$. 
\end{proof}

\subsection{Simply connected case}\label{subsec:pi10}
If $G$ is not only semi-simple but also simply connected, the discussion of pre-quantization simplifies. One has $H^1(G,\Z)=H^1_G(G,\Z)=0,\ H^2(G,\Z)=H^2_G(G,\Z)=0$,
while 
\[H^3(G,\Z)=H^3_G(G,\Z)=\Z^N,\] 
where $N$ is the number of simple factors of $G$. As a consequence, the specific choice of the $G$-DD bundle 
$\A$ with $\on{DD}_G(\A)$ an integral lift of $[\eta_G]$ is unimportant: Any two choices are Morita isomorphic, where the 
Morita isomorphism is unique up to 2-isomorphism. Furthermore, it is automatic that $\A$ is multiplicative, 
with $\F_{\on{Mult}}$ unique up to 2-isomorphism. 

Suppose that $(M,\om,\Phi)$ is a q-Hamiltonian $G$-space. A pre-quantization amounts to an integral lift of the equivariant class $[(\om,\eta_G)]\in H^3_G(\Phi,\R)$. Since $H^2_G(G,\Z)=0$, the map $H^2_G(M,\Z)\to H^3_G(\Phi,\Z)$
is injective. Hence, if a pre-quantization exists, then it is unique up to a flat line bundle.  
The situation simplifies further, due to the following fact: 
\begin{lemma}\label{lem:dropg}
  Suppose $G$ is simply connected, and that $X$ is a $G$-space.
  Then the natural map $H^p_G(X)\to H^p(X)$ (arbitrary coefficients)
  is an isomorphism for $p\le 2$, and is injective for $p=3$.
  Moreover, for any equivariant map $\Phi\colon X\to G$ the map
  $H^3_G(\Phi)\to H^3(\Phi)$ is an isomorphism.
\end{lemma}
\begin{proof} (See \cite{kre:pr}) The first claim follows from the Serre exact sequence for
  the fibration $X_G\to BG$. The second claim follows from the 5-Lemma
  applied to the maps between the long exact sequences for
  $H^3_G(\Phi)$ and $H^3(\Phi)$, using that the map $H^3_G(G)\to
  H^3(G)$ is an isomorphism.
\end{proof}
Hence, for $G$ simply connected we may define a pre-quantization to be an integral lift of the non-equivariant 
class $[(\om,\eta)]\in H^3(\Phi,\R)$ -- the extension to an equivariant class is automatic. 
The integrality of the relative class $[(\om,\eta)]$ can be expressed in terms of integration. 
(See Appendix \ref{app:rel}.)
Thus, a q-Hamiltonian $G$-space at integral level $k$ is pre-quantizable if and only if for every map $f\colon \Sigma\to M$ from a closed oriented surface $\Sigma$,          
          \[ \int_\Sigma f^*\omega-\int_{\Sigma\times[0,1]}h^*\eta\in\Z. \]
Here $h\colon \Sigma\times [0,1]\to G$ is any homotopy between $h_0=\Phi\circ f$ and $h_1=e$ (the constant map 
      to $e$). For instance, if $M$ is 2-connected, then $f$ itself is homotopic to a constant map, hence the condition is satisfied. 
      
Let $G=G_1\times\cdots \times G_N$ be the
decomposition into simple factors. For each $k=(k_1,\ldots,k_n)\in\R^n$, 
we denote by $B_k$ the unique invariant symmetric bilinear form on $\g$
whose restriction to $\g_i$ is the $k_i$-th multiple of the basic
inner product of $\g_i$. That is, if $T\subset G$ is a maximal torus,
and $\alpha^\vee\in\t\cap\g_i$ is a short co-root for $\g_i$, then
$B_k(\alpha^\vee,\alpha^\vee)=2 k_i$. Then the 3-form
$\eta\in\Om^3(G)$ defined using $B=B_k$ is integral if and only if all
$k_i$ are integers. If the inner product in our definition of
q-Hamiltonian $G$-space is $B=B_k$ (where $k_i>0$), we refer to $k$ as the
\emph{level} of our theory. Thus, a necessary condition for the
pre-quantizability of a q-Hamiltonian $G$-space at level $k$ is the
integrality of the level.

Let us examine the pre-quantizability of our main examples. 

\begin{proposition}
Suppose $G$ be a simply connected Lie group.
\begin{enumerate} 
\item (Conjugacy classes.) 
Let $T$ be a maximal torus in $G$, and let $\Lambda\subset\t$ be the integral lattice.  
The $G$-conjugacy class $\Co=G.\exp(\xi)$ for $\xi\in \t$ is
pre-quantized at level $k$ if and only if $B_k(\xi,\lambda)\in\Z$ for all $\lambda\in\Lambda$.  
\item (Moduli spaces.) The q-Hamiltonian $G^r$-space $M(\Sig_h^r)$ for $r\ge 1$ admits a pre-quantization
for any integer level $k$. 
\item (The sphere and $\mathbb{H} P(n)$.)
The q-Hamiltonian $\SU(n)$-space  for $n\ge 2$, and the q-Hamiltonian 
$\on{Sp}(n)$-space $\mathbb{H} P(n)$ for $n\ge 2$ admit pre-quantizations for any integer level $k$.
\end{enumerate} 
In all of these examples, the pre-quantization is unique up to 2-isomorphism.
\end{proposition}
\begin{proof}
Let $M$ be any of the q-Hamiltonian $G$-spaces in (b),(c).  Let $y\in H^3(G,\Z)$ 
be the unique integral lift of $[\eta]$. Since $\Phi^*[\eta]=0$, the class $\Phi^*y$ is torsion. 
But $H^3(M,\Z)$ is torsion-free for the spaces in (b), (c).  it follows that $\Phi^*y=0$, so that $y$ is the image 
of a class in $H^3(\Phi,\Z)$. Since $H^2(M,\Z)=0$ for the spaces in (b), (c), the long exact sequence in relative cohomology shows that the class $y$ is unique. It is automatic that its image in $H^3(\Phi,\R)$ is $[(\om,\eta)]$ (this being the unique pre-image of $[\eta]$).  For (a), see \cite{me:su}. The uniqueness part follows since in all of these examples, 
$H^2(M,\Z)$ has no torsion. 
\end{proof}

Part (a) shows that the pre-quantized conjugacy classes at level $k$ are indexed by the set $\Lambda^*_k$ 
of level $k$ weights (see Appendix \ref{app:fus}).  

\begin{example}
  Let $\Co_1,\ldots,\Co_r\subset G$ be pre-quantized conjugacy
  classes, and $h\ge 0$. Then the fusion product $D(G)^h\times
  \Co_1\times\cdots\times \Co_r$ is pre-quantized. One hence obtains a
  pre-quantization of its symplectic quotient, the moduli space of
  flat $G$-bundles $M(\Sig_h^r,\Co_1,\ldots,\Co_r)$. (In most cases,
  the symplectic quotient has orbifold singularities, or even worse
  singularities. Pre-quantization of singular symplectic quotients is
  discussed in \cite{me:si}.)
\end{example}

\section{The quantization map}
In this Section, we will define the quantization of a pre-quantized q-Hamiltonian $G$-space, as a push-forward 
in twisted equivariant $K$-homology. We then specialize to the case that $G$ is simply connected, 
identifying the target of the quantization map with the fusion ring (Verlinde algebra). We begin by reviewing 
the quantization of ordinary Hamiltonian $G$-spaces. 

\subsection{Quantization of Hamiltonian $G$-spaces}\label{subsec:hamquant}
The quantization of Hamiltonian $G$-spaces in symplectic geometry, using $\Spin_c$-Dirac operators, was 
introduced by Guillemin in \cite{gu:re}. Suppose $(M,\om)$ is a compact symplectic manifold, 
with a Hamiltonian action of a compact Lie group $G$, with moment map $\Phi\colon M\to \g^*$. Let $L\to M$ be a $G$-equivariant pre-quantum line bundle
(i.e.~ the equivariant Chern class of $L$ is an integral lift of the class of the equivariant symplectic form). 
The choice of a $G$-invariant compatible almost complex structure on $M$ determines a $\Spin_c$-structure, with spinor bundle $\ca{S}$. Let $\dirac_L$ be the $\Spin_c$-Dirac operator with coefficients in $L$, that is, the 
Dirac operator for the twisted spinor module $\ca{S}\otimes L$. We define 
the quantization of $M$ to be the $G$-equivariant index of this Dirac operator:
\[ \ca{Q}(M)=\on{index}_G(\dirac_L)\in R(G).\]
As mentioned in the introduction, this quantization procedure is well-behaved under products and conjugation, 
and it satisfies a `quantization commutes with reduction' principle. The element $\ca{Q}(M)$ 
may be computed using the equivariant index theorem of Atiyah-Segal-Singer. 

In order to generalize to q-Hamiltonian spaces, it is convenient to rephrase the quantization procedure 
in terms of K-theory. The spinor bundle defines a $G$-equivariant Morita trivialization 
\begin{equation}\label{eq:spinc}
 (p,\ca{S})\colon (M,\Cl(TM))\da (\pt,\C),\end{equation} 
where $p\colon M\to \pt$ is the map to a point, and where the canonical anti-automorphism of the Clifford algebra is used to turn $\ca{S}$ into a right module. The pre-quantum line bundle may be viewed as a Morita morphism between trivial 
DD bundles, 
\begin{equation}\label{eq:preqs}
(p,L)\colon (M,\C)\da (\pt,\C).\end{equation}
Its tensor product with \eqref{eq:spinc} is the $L$-twisted $\Spin_c$-structure, 
$(p,\ca{S}\otimes L)\colon (M,\Cl(TM))\da (\pt,\C)$, and defines
a push-forward map in twisted $K$-homology, 
\[ p_*=K_0^G(p,\ca{S}\otimes L)\colon  K_0^G(M,\Cl(TM))\to K_0^G(\pt)\cong R(G). \]
The quantization is the push-forward of the fundamental class under this map: 
\[\ca{Q}(M)=p_*([M]).\] 
Note that this reformulation of the quantization no longer mentions the 
$\Spin_c$-Dirac operator.

\subsection{The canonical twisted $\on{Spin}_c$-structure}
Let $G$ be a compact connected Lie group. q-Hamiltonian $G$-spaces need not admit $\Spin_c$-structures in general. However, as shown in 
\cite{al:ddd} they carry a canonical `twisted' $\Spin_c$-structure. The definition involves a 
distinguished $G$-DD bundle $\A^{\Spin}\to G$, originally defined in \cite{at:twi} in terms of a bundle of projective Hilbert spaces. 

It can be defined as the pull-back, under the adjoint representation $\on{Ad}\colon G\to \SO(n)$
(using an orthonormal basis to identify $\g\cong \R^n$) of an $\SO(n)$-equivariant DD bundle 
$\A_{\SO(n)}\to \SO(n)$. The main properties of $\A_{\SO(n)}$ are as follows: i) $\A_{\SO(n)}$
is the restriction of $\A_{\SO(n+1)}$ under the inclusion $\SO(n)\hra \SO(n+1)$, and 
ii) if $n\ge 3$, the $H^3(\cdot,\Z)$-component and the $H^1(\cdot,\Z_2)$-component of the 
DD class of $\A_{\SO(n)}$ restrict to generators of
$H^3(\SO(n),\Z)=\Z$ respectively $H^1(\SO(n),\Z_2)=\Z_2$. As shown in \cite{al:ddd} the bundle $\A_{\SO(n)}$ 
is multiplicative, with a canonical (up to 2-isomorphism) multiplication morphism. 
Hence $\A^{\on{Spin}}\to G$ is multiplicative as well: 
\[ (\on{Mult},\F_{\on{Mult}})\colon (G\times G,\A^{\on{Spin}}\times \A^{\on{Spin}})\da (G,\A^{\on{Spin}}).\]
Let $\F_i,\F_{\on{Inv}},\F_{\exp}$ be the morphisms obtained from $\F_{\on{Mult}}$, as in Section 
\ref{subsec:mult} above. 
\begin{remark}
If $G$ is simple and simply connected, the Dixmier-Douady class $\on{DD}_G(\A^{\on{Spin}})\in H^3_G(G,\Z)=\Z$ is $\cox$ times the generator, where $\cox$ is the dual Coxeter number of $G$. 
\end{remark} 
For the next result, we recall \cite{al:mom} that q-Hamiltonian $G$-spaces $M$ for
compact, connected groups $G$ are always even-dimensional. Hence $\Cl(TM)$ is a $G$-DD bundle.

\begin{theorem}\cite{al:ddd}\label{th:ddd}
Let $G$ be a compact, connected Lie group. 
For any q-Hamiltonian $G$-space $(M,\om,\Phi)$ there is a distinguished 2-isomorphism class of 
$G$-equivariant Morita morphisms
\[ (\Phi,{\ca{S}})\colon (M,\Cl(TM))\da (G,\A^{\on{Spin}}),\]
We will refer to such a Morita morphism as a \emph{standard twisted $\Spin_c$-structure}
of $(M,\om,\Phi)$. If $G=\{1\}$ (so hat $(M,\om)$ is symplectic and $\A^{\on{Spin}}=\C$),  
the standard twisted $\Spin_c$-structure is the standard $\Spin_c$-structure 
$(M,\Cl(TM))\da (\pt,\C)$ of a symplectic manifold.  The canonical twisted $\Spin_c$-structures have 
the following properties: 
\begin{enumerate}
\item {\bf Conjugates.} If $(\Phi,{\ca{S}})$ is the standard twisted $\Spin_c$-structure for the q-Hamiltonian $G$-space $(M,\om,\Phi)$, then the composition 
\[ (\on{Inv},\F_{\on{Inv}})\circ (\Phi,\ol{\ca{S}})\colon (M,\Cl(TM))=(M,\Cl(TM)^{\on{op}})\da (G,\A^{\on{Spin}})\]
defines a standard twisted $\Spin_c$-structure of its conjugate $(M^*,\om^*,\Phi^*)$. 
\item {\bf Fusion.} Suppose $(M,\om,\Phi)$ is a q-Hamiltonian $G\times G$-space, and 
$(\Phi,{\ca{S}})\colon (M,\Cl(TM))\da (G\times G,\A^{\on{Spin}}\times\A^{\on{Spin}})$ its standard twisted $\Spin_c$-structure. 
Then 
\[ (\Phi_{\on{fus}},{\ca{S}}_{\on{fus}}):=(\Mult,\F_{\Mult})\circ (\Phi,{\ca{S}})\colon (M,\Cl(TM))\da (G,\A^{\on{Spin}})\]
is a standard twisted $\Spin_c$-structure for the fusion $(M_{\on{fus}},\om_{\on{fus}},\Phi_{\on{fus}})$. 
\item{\bf Exponentials.} 
Suppose $(M,\om_0,\Phi_0)$ is a Hamiltonian $G$-space whose moment map image $\Phi_0(M)\subset \g^*\cong\g$ is contained in a connected open neighborhood of $0\in\g$ over which the exponential map is $1-1$. 
If $(\Phi_0,{\ca{S}}_0)$ is a standard $\Spin_c$-structure for 
$(M,\om_0,\Phi_0)$ then 
\[ (\Phi,{\ca{S}})=(\exp,\F_{\exp})\circ (\Phi_0,{\ca{S}}_0)\colon 
(M,\Cl(TM))\da (G,\A^{\on{Spin}})
\]
is a standard twisted $\Spin_c$-structure for its exponential $(M,\om,\Phi)$. 
\end{enumerate}
\end{theorem}

The twisted $\Spin_c$-structure is compatible with reduction, in the following sense. 
Suppose $(M,\om,\Phi)$ is a q-Hamiltonian $G$-space, with $e$ is a regular value of the moment map, and with 
$G$ acting freely on $Z=\Phi^{-1}(e)$. Let $M_\red=Z/G$ be the reduced space. Using the isomorphism $TM|_Z=TM_\red
\oplus (\g\oplus \g)$ one obtains a Morita morphism, $(Z,\Cl(TM)|_Z)\da 
(M_\red,\Cl(TM_\red))$ covering the identity. Then the following diagram of Morita morphisms commutes 
up to 2-isomorphism:
\[ \xymatrix{
 {(M,\Cl(TM))} \ar@{-->}[r] &{(G,\A^{\Spin})}\\ 
 {(Z,\Cl(TM)|_Z)}\ar@{-->}[r]\ar@{-->}[d] \ar@{-->}[u] & 
%\ar@[r]
{(\{e\},\A^{\on{Spin}}|_{\{e\}}})\ar@{-->}[d]\ar@{-->}[u]
\\
{(M_\red,\Cl(TM_\red))}\ar@{-->}[r] & {(\pt,\C)}
}\]

Here the horizontal maps are the canonical (twisted) $\Spin_c$-structures for  $M,M_\red$.

\subsection{Quantization}
Suppose $(M,\om,\Phi)$ is a compact pre-quantized q-Hamiltonian $G$-space. 
The canonical twisted $\Spin_c$-structure and the pre-quantization 
\[ \begin{split}(\Phi,{\ca{S}})&\colon (M,\Cl(TM))\da (G,\A^{\Spin})\\
(\Phi,\E)&\colon (M,\C)\da (G,\A^{\on{Preq}})
\end{split}\]
are the counterparts to \eqref{eq:spinc} and \eqref{eq:preqs}. Tensoring, we obtain a 
$G$-equivariant Morita morphism $(\Phi,{\ca{S}} \otimes \E )\colon (M,\Cl(TM))\da (G,\A^{\on{Preq}}\otimes \A^{\Spin})$, 
and a push-forward map 
\[ \Phi_*=K_0^G(\Phi,\ca{S}\otimes \E)
\colon K_0^G(M,\Cl(TM))\to K_0^G(G,\A^{\on{Preq}}\otimes \A^{\Spin}).\]
\begin{definition}
For a compact pre-quantized q-Hamiltonian $G$-space $(M,\om,\Phi)$ we define its quantization as the push-forward in twisted $K$-homology, 
\[ \ca{Q}(M):=\Phi_*([M])\in K_0^G(G,\A^{\on{Preq}}\otimes \A^{\Spin}).\] 
\end{definition}

\begin{remarks}
\begin{enumerate}
\item Given an equivariant vector bundle $E\to M$, one may
cap $[M]$ with the corresponding equivariant $K$-theory class,
 thus
defining invariants \[\ca{Q}(M,E):=\Phi_*([E]\cap [M])\in K_0^G(G,\A^{\on{Preq}}\otimes \A^{\Spin}).\]
(Equivalently, $[E]\cap [M]$ is the image of $[E]$ under the Poincar\'{e} duality isomorphism 
$K^0_G(M)\to K_0^G(M,\Cl(TM))$.)  

\item In the symplectic case (Section \ref{subsec:hamquant}), one obtains a sequence $\ca{Q}_r(M)
=\on{index}(\dirac_{L^r})\in R(G)$ by replacing 
$L$ with its $r$-th tensor power. For $r>0$ this amounts to replacing $\om$ with its 
$r$-th multiple. The limit $r\to \infty$ is referred to as the semi-classical limit, with 
$r^{-1}$ playing the role of Planck's constant $\hbar$. For pre-quantized q-Hamiltonian spaces, 
one similarly obtains a sequence of elements $\ca{Q}_r(M)\in R_{kr}(G)$, by raising 
$(\Phi,\E)\colon (M,\C)\da (G,\A^{\on{Preq}})$ to its $r$-th tensor power. 
\end{enumerate}
\end{remarks}

Suppose that $\A^{\on{Preq}}$ is multiplicative. Let $\F_{\on{Mult}}^{\on{Preq}},\ \F_{\on{Mult}}^{\Spin}$ be the Morita morphisms defining the multiplicative structures 
for $\A^{\on{Preq}},\ \A^{\Spin}$. Then $\F_{\on{Mult}}=\F_{\on{Mult}}^{\on{Preq}}\otimes \F_{\on{Mult}}^{\Spin}$
defines a multiplicative structure for $\A=\A^{\on{Preq}}\otimes \A^{\Spin}$, and 
as explained in Section \ref{subsec:mult} the group 
\begin{equation}\label{eq:R} \mathsf{R}_\bullet=K_\bullet^G(G,\A^{\on{Preq}}\otimes \A^{\Spin})\end{equation}
is a ring with involution. 
Similar to the quantization of Hamiltonian $G$-spaces, 
we have: 
\begin{theorem}[Properties of q-Hamiltonian quantization]\label{th:prop}
Let $\A^{\on{Preq}}\to G$ be a multiplicative $G$-DD bundle whose class defines an integral lift of 
$[\eta_G]$, and let $\mathsf{R}$ be the ring \eqref{eq:R}.
\begin{enumerate}
\item {\bf Conjugation.} Let $(M,\om,\Phi)$ be a q-Hamiltonian $G$-space,  
pre-quantized by $(\E,\A^{\on{Preq}})$, and 
give $M^*$ the conjugate pre-quantization $\E^*$ (cf.~ Proposition \ref{prop:conjugates}). 
Then  
\[\ca{Q}(M^*)=\ca{Q}(M)^*\]
using conjugation in the ring $\mathsf{R}$. 
\item {\bf Fusion.} Let $(M,\om,\Phi)$ be a q-Hamiltonian $G\times G$-space, pre-quantized
by $(\E,\A^{\on{Preq}}\times\A^{\on{Preq}})$. Let $M_{\on{fus}}$ be pre-quantized by 
$(\E_{\on{fus}},\A^{\on{Preq}})$ (cf.~ Proposition \ref{prop:fusions}). Then 
$\ca{Q}(M_{\on{fus}})$ is the image of $\ca{Q}(M)$ under the product map 
$\mathsf{R}\otimes \mathsf{R}\to \mathsf{R}$. 
  In particular, if $(M_i,\om_i,\Phi_i),\ i=1,2$ are pre-quantized by $(\E_i,\A^{\on{Preq}})$ then 
\[ \ca{Q}(M_1\fus M_2)=\ca{Q}(M_1)\ca{Q}(M_2).\]
\item{\bf Exponentials.} Let $(M,\om_0,\Phi_0)$ be a compact pre-quantized Hamiltonian $G$-space such that 
$\Phi_0(M_0)\subset\g^*\cong\g$
is contained in a contractible open neighborhood of $0$ where $\exp$ is 1-1.
Let $(M,\om,\Phi)$ be its exponential, with the resulting pre-quantization $(\E,\A^{\on{Preq}})$ (cf.~ Proposition
\ref{prop:exxponentials}). 
Then $\ca{Q}(M)$ is the image of $\ca{Q}(M_0)\in R(G)$ under the ring homomorphism 
$R(G)\to \mathsf{R}$. 
\end{enumerate}
\end{theorem}
\begin{proof}
All of these facts are direct consequences of the functoriality properties established earlier. 
Consider for example the proof of (b). We have 
\[ (\Phi_{\on{fus}},\E_{\on{fus}})=(\on{Mult},\F_{\on{Mult}}^{\on{Preq}})\circ (\Phi,\E),\]
by definition of the pre-quantization of $M_{\on{fus}}$, and 
\[ (\Phi_{\on{fus}},{\ca{S}}_{\on{fus}})=(\on{Mult},\F_{\on{Mult}}^{\Spin})\circ (\Phi,{\ca{S}}),\]
by Theorem \ref{th:ddd}. Hence 
\[ (\Phi_{\on{fus}},\E_{\on{fus}} \otimes \ca{S}_{\on{fus}})
=(\on{Mult},\F_{\on{Mult}})\circ 
(\Phi,\E\otimes \ca{S}).\]
Applying the K-homology functor, this gives $(\Phi_{\on{fus}})_*=\on{Mult}_*\circ \Phi_*$, proving (b). Parts (a) and (c) are obtained similarly.  
\end{proof}

\begin{remark}
The same argument gives generalizations of these results to $\ca{Q}(M,E)$, for Hermitian vector bundles $E$. 
\end{remark}

\begin{remark}
The quantization of compact Hamiltonian $G$-spaces satisfies a Guillemin-Sternberg `quantization commutes with reduction' 
principle, see \cite{me:si,me:sym}. There is a similar result for q-Hamiltonian spaces, at least for $G$ simply connected. 
In more detail, suppose $(M,\om,\Phi)$ is a compact pre-quantized q-Hamiltonian $G$-space at level $k$, and let 
$\ca{Q}(M)\in R_k(G)$ be its quantization. The level $k$ fusion ring has an additive basis 
$\tau_\mu,\ \mu\in\Lambda^*_k$ indexed by the set of level $k$ weights, see Appendix  \ref{app:fus}, so that 
$\ca{Q}(M)=\sum N(\mu)\tau_\mu$ with multiplicities $N(\mu)\in\Z$. Let $\ca{Q}(M)^G=N(0)$ be the multiplicity of $\tau_0$. 
Suppose that $e$ is a regular value of $\Phi$, and that $G$ acts freely on $\Phinv(e)$. The `quantization commutes with reduction' principle asserts that 
\begin{equation}\label{eq:quantred} \ca{Q}(M)^G=\ca{Q}(M\qu G).\end{equation}
One may drop the regularity assumptions on the moment map, provided the quantization of $M\qu G$ is defined using `partial desingularization' of the quotient, as in \cite{me:si}. 
In terms of the correspondence between q-Hamiltonian $G$-spaces and Hamiltonian $LG$-spaces, and 
`defining' $\ca{Q}(M)$ in terms of fixed point contributions (Theorem \ref{th:fix} below), 
the equality \eqref{eq:quantred} was proved in \cite{al:fi}. Unfortunately, the proof in \cite{al:fi} is fairly complicated, and one would certainly prefer a more conceptual argument. 
The appropriate generalization to non-simply connected groups is an open question as well. 
\end{remark}

\section{Localization}

For the remainder of this paper, we assume that $G$ 
is compact and simply connected. Throughout, we fix a decomposition into simple factors $G=G_1\times\cdots \times G_N$. 
For $l\in \Z^N$, we denote by $\A^{(l)}$ any 
$G$-DD bundle with $\on{DD}_G(\A^{(l)})\in H^3_G(G,\Z)\cong \Z^N$ given by 
$l$. For example, $\A^{\Spin}$ is a possible choice for $\A^{(\cox)}$. 
Furthermore, we will fix a maximal torus $T$, with Lie algebra $\t$ and lattice 
$\Lambda\subset\t$, as well as a fundamental Weyl chamber. Let $\alpha_i\in\Lambda^*=\on{Hom}(\Lambda,\Z)$ 
be the corresponding simple roots, $\alpha_i^\vee\in\Lambda$ the simple co-roots, and $\varpi_i\in\Lambda^*$ the fundamental 
weights. 

For a pre-quantized q-Hamiltonian $G$-space 
at level $k$, the invariants $\ca{Q}(M,E)$  are elements of the level $k$ fusion 
ring. In the Appendix (cf.~ Definition \ref{def:fus}) we define the fusion ring as a quotient 
$R_k(G)=R(G)/I_k(G)$, where the level $k$ fusion ideal $I_k(G)$ consists of characters vanishing at all 
\[ t_\lambda=\exp(B_{k+\cox}^\sharp(\lambda+\rho)),\ \ \lambda\in\Lambda^*_k.\]
Here $\Lambda^*_k$ are the level $k$ weights, i.e. those weights such that $B_{k}^\sharp(\lambda)$ lies in 
the fundamental alcove. The evaluation map $\on{ev}_{t_\lambda}\colon R(G)\to \C,\ \chi\mapsto \chi(t_\lambda)$ descends to a 
map $\on{ev}_{t_\lambda}\colon R_k(G)\to \C,\ \tau\mapsto \tau(t_\lambda)$. By inverse Fourier transform
(see Proposition \ref{prop:orth}), any $\tau\in R_k(G)$ may be recovered 
from these values. Our goal in this Section is to compute the numbers
\[\ca{Q}(M,E)(t_\lambda):=\on{ev}_{t_\lambda}\ca{Q}(M,E)\]
by localization to the fixed point set of $t_\lambda$.

\subsection{Restriction of $\A^{(l)}\to G$ to the maximal torus}\label{subsec:res}
The localization procedure will involve the restriction of $\A^{(k+\cox)}$ to the
maximal torus $T\subset G$. Since the pull-back map $H^3(G,\Z)\to H^3(T,\Z)$
is the zero map, it is immediate that $\A^{(l)}|_T$ is
Morita trivial, for any $l\in\Z^N$. However, the pull-back is not \emph{$T$-equivariantly} 
Morita trivial. The following fact is proved in \cite[Proposition 3.1]{me:con}. 
\begin{lemma} \label{lem:ma}
The map 
\[H^3_G(G,\Z)\to H^3_T(T,\Z)=H^3(T,\Z)\oplus H^1(T,\Z)\otimes H^2_T(\pt,\Z)\] 
takes values in the second summand $H^1(T,\Z)\otimes H^2_T(\pt,\Z)
=\Lambda^*\otimes\Lambda^*$, and takes $l\in \Z^N=H^3_G(G,\Z)$  
to 
\begin{equation}\label{eq:lo}
 -\sum_{i,j} B_l(\alpha_i^\vee,\alpha_j^\vee) \varpi_i\otimes\varpi_j
 =-\sum_i \varpi_i\otimes B_l^\flat(\alpha_i^\vee).
\end{equation}
\end{lemma}

%\begin{proof}
%Assume that $G$ is simple. Since $H^3_G(G,\Z), H^3_T(T,\Z)$ have no torsion, we may pass to real 
%coefficients and work with 
%differential forms. The element $l\in \Z^N=H^3(G,\Z)$ is represented by the equivariant 3-form
%$\eta_G\in \Om^3_G(G)$, defined using the inner product $B_l$ on $\g$. The map $\Om^3_G(G)\to \Om^3_T(T)$ takes 
%$\eta_G$ to $B_l(\theta_T,\cdot)\in \Om^3_T(T)=
%\Om^3(T)\oplus \Om^1(T)\otimes \t^*$, where $\theta_T\in \Om^1(T,\t)$ is the Maurer-Cartan form on $T$. 
%Thinking of $\varpi_i$ as 1-forms on $T$, we have $\theta_T=\sum_i \varpi_i\otimes \alpha_i^\vee$, hence 
%$B_l(\theta_T,\cdot)=\sum_i \varpi_i\otimes B_l(\alpha_i^\vee,\cdot)
%=\sum_{i,j} B_l(\alpha_i^\vee,\alpha_j^\vee) \varpi_i\otimes\varpi_j$. 
%\end{proof}

For a level $l\in \Z^N$ with $l_i>0$, we define a finite subgroup $T_l\subset T$ by 
\[ T_l=B_l^\sharp(\Lambda^*)/\Lambda\]
(see Appendix \ref{app:fus}). Notice that the elements $t_\lambda$ are contained in $T_{k+\cox}$. 

\begin{proposition}\label{prop:restrict}
Suppose $l\in \Z^N$ with $l_i>0$. There is a $T_l$-equivariant Morita trivialization 
\[ (p,\D)\colon (T,\A^{(l)}|_T)\da (\pt,\C)\]
such that $\D|_{\{e\}}$ extends to a $G$-equivariant Morita trivialization of $\A^{(l)}|_e$. 
This property determines $\D$  up to a line bundle over $T$ with trivial $T_l$-action.  
\end{proposition}
\begin{proof}
The map 
\[ H^3_G(G,\Z)\to H^3_{T_r}(T,\Z)=H^3(T,\Z)\oplus H^1(T,\Z)\otimes H^2_{T_r}(\pt,\Z)\] 
is given by the map to $H^1(T,\Z)\otimes H^2_T(\pt,\Z)=\Lambda^*\otimes\Lambda^*$ 
followed by the quotient map $\Lambda^*\otimes (\Lambda^*/B_l^\flat(\Lambda))$. 
Lemma \ref{lem:ma} shows that the element $l\in \Z^N=H^3_G(G,\Z)$ maps to $0$ under this map. 
This shows the existence of a $T_l$-equivariant Morita trivialization $(p,\D)
\colon (T,\A^{(l)}|_T)\da (\pt,\C)$. It is unique up to a $T_{l}$-equivariant line bundle over $T$.
By twisting with a 1-dimensional representation of $T_{l}$ on $\C$, if necessary, 
we can arrange that $\D|_e$ is the unique $G$-equivariant Morita trivialization
of $\A^{(l)}|_e$. The remaining ambiguity is a line bundle over $T$ with trivial $T_l$-action. 
\end{proof}
Proposition \ref{prop:restrict} will suffice for the purposes of our localization formula. 
However, it seems natural to ask to what extent the remaining $H^2(T,\Z)$-ambiguity in the 
Morita trivialization can be removed. It turns out that there is a canonical choice 
in some cases (e.g.~ for $l$ even). The remainder of this Section will investigate this 
question. 
 
Equivalence classes of (non-equivariant) Dixmier-Douady bundles over $G$ with Morita trivializations over 
$T$ are classified by the relative cohomology group $H^3(\iota_T,\Z)=H^3(G,T,\Z)$. The long exact sequence in relative cohomology gives a short exact
sequence
\[ 0\lra H^2(T,\Z)\to H^3(G,T,\Z)\to H^3(G,\Z)\lra 0\]
We are interested in choices of splittings $H^3(G,\Z)\to H^3(G,T,\Z)$ of this sequence.
  
Since neither $H^3(G,\Z)=\Z^N$ nor $H^2(T,\Z)=\wedge^2_\Z \Lambda^*$ have  
torsion, the group $H^3(G,T,\Z)$ has no torsion. Hence no information is lost by 
passing to real coefficients, and we may work with differential forms (using de Rham cohomology). 
Over the real numbers, the sequence splits by the map 
\begin{equation}\label{eq:splitR} H^3(G,\R)\to H^3(G,T,\R),\ [\beta]\mapsto [(0,\beta)].\end{equation}
Hence $H^3(G,T,\R)\cong H^2(T,\R)\oplus H^3(G,\R)=
\wedge^2_\R \Lambda^*\oplus \R^N$. 
This defines an inclusion
\begin{equation}\label{eq:inclus}
H^3(G,T,\Z)\hra \wedge^2_\R \Lambda^*\oplus \R^N,\end{equation}
and we are interested in its image. We may assume that $G$ is simple (so that $N=1$).
The wedge products $\varpi_i\wedge \varpi_j,\ i<j$ are a 
$\Z$-basis of $H^2(T,\Z)=\wedge^2_\Z\Lambda^*$. Define an element 
\[ \varsigma=\sum_{i<j} B(\alpha_i^\vee,\alpha_j^\vee) \varpi_i\wedge\varpi_j\in\wedge^2_\Z \Lambda^*,\]
where $B$ is the basic inner product on $\g$. Note that $\varsigma$ depends not only on the choice of Weyl chamber, 
but also 
on the choice of ordering of the 
simple roots. 

\begin{proposition}
The image of the map \eqref{eq:inclus}
consists of all $(\delta+\f{s}{2}\varsigma,s)$ such that 
$s\in\Z$ and $\delta\in \wedge^2_\Z\Lambda^*$. We hence obtain a splitting $H^3(G,\Z)=\Z\to H^3(G,T,\Z)
\subset \wedge^2_\R \Lambda^*\oplus \R$ by 
\[ s\mapsto \Big(\f{s}{2}\varsigma,\ s\Big).\]
\end{proposition}
\begin{proof}
We may regard the $\varpi_i$ as translation invariant 1-forms on $T$, hence $\varpi_i\wedge \varpi_j$ 
are viewed as translation invariant 2-forms on $T$. A relative 3-class
\begin{equation}\label{eq:anelement}
\Big(\sum_{i<j} t_{ij}\varpi_i\wedge\varpi_j,\ s\Big)\in \wedge^2_\R \Lambda^*\oplus \R
\end{equation} 
is integral if and only if for all maps $\lambda\colon \Sigma\to T$ from a closed surface $\Sigma$, and all homotopies 
$h\colon \Sigma\times [0,1]\to G$ between $h_0=\iota_T\circ \lambda$ and the constant map 
$h_1=e$, 
\[ \int_\Sigma \lambda^*\sum_{i<j} t_{ij}\varpi_i\wedge\varpi_j-s\int_{\Sigma\times[0,1]} 
h^*\eta\in\Z; \]
here $\eta\in\Om^3(G)$ is the Cartan 3-form defined by the 
basic inner product (so $[\eta]\in H^3(G,\R)$ is the image of the generator of $H^3(G,\Z)=\Z$). 
We have $H_2(T,\Z)=\wedge^2_\Z\Lambda=
\wedge^2_\Z[\alpha_1^\vee,\ldots,\alpha_n^\vee]$, where 
the class $\alpha_i^\vee\wedge\alpha_j^\vee,\ i<j$ is represented by the 2-cycle   
\[ \lambda\colon \Sigma=(\R/\Z)^2\to T,\ (u_1,u_2)\mapsto 
\exp(u_1\alpha_i^\vee+u_2\alpha_j^\vee).\]
It suffices to check the condition for these generators. 
Let $\gamma_1,\gamma_2\colon \R/\Z\times [0,1]\to  G$ be continuous retractions 
of the loops $\exp(u_1 \alpha_i^\vee),\ 
\exp(u_2 \alpha_j^\vee)$ onto $e$: That is, 
\[ \gamma_1(\cdot,0)=\gamma_1(\cdot,1)=\gamma_1(u_1,1)=e;\ \ 
\gamma_1(u_1,0)=
\exp(u_1\alpha_i^\vee)\]
and similarly for $\gamma_2$. Then 
\[ h\colon \Sigma\times [0,1]\to G,\ \ (u_1,u_2,t)\mapsto \gamma_1(u_1,t)\gamma_2(u_2,t)\]
is a homotopy from $\lambda$ to the constant map to $e$.  
We have 
\[ \int_{\Sigma}\lambda^*(\sum_{r<s}t_{rs} \varpi_r\wedge\varpi_s)
= \int_{\Sigma} t_{ij} \d u_1\wedge \d u_2=t_{ij}.\]
On the other hand, 
\[ h^*\eta=\gamma_1^*\eta+\gamma_2^*\eta+\hh \d B(\gamma_1^*\theta^L,\gamma_2^*\theta^R).\]
Integrating over $\Sigma \times [0,1]$, the first two terms do not contribute since they do not
contain $\d u_2$, respectively $\d u_1$. The integral of the last term
reduces, by Stokes' theorem, to an integral over 
$\Sigma\times\{0\}\subset \Sigma\times [0,1]$:
\[ \int_{\Sigma\times [0,1]} h^*\eta=\hh
\int_{\Sigma}B(\gamma_1(\cdot,0)^*\theta^L,\gamma_2(\cdot,0)^*\theta^R)
=\hh B(\alpha_i^\vee,\alpha_j^\vee).\]
We conclude that the pairing is given by $t_{ij}-\f{s}{2} B(\alpha_i^\vee,\alpha_j^\vee)$.  
Hence, $s$ and all of these numbers must be integers.
\end{proof}

\begin{remark}
As a consequence, we obtain a distinguished Morita trivialization of $\A^{(l)}$ provided the level $l$ 
is even, or more generally if $B_l(\alpha_i^\vee,\alpha_j^\vee)\in 2\Z$ for $i<j$. 
Indeed, the Proposition shows that under this assumption, $(0,l)\in \wedge^2_\R\Lambda^*\oplus \R\cong H^3(G,T,\R)$ 
is integral, because (for $G$ simple) $\f{l}{2}\varsigma$ is integral in that case. We note that 
the basic inner product takes on even values on $\Lambda$ in the following cases:
$G=\SU(2), G=\Spin(5)$ or $G=\on{Sp}(n),\ n\ge 3$. Hence, for these groups the Morita trivialization of $\A^{(l)}|_T$ 
is canonical up to 2-isomorphism, for all $l\in\Z$. 
\end{remark}

\subsection{The evaluation map $\on{ev}_t\colon R_k(G)\to \C$}
Let $k=(k_1,\ldots,k_N)$ with $k_i\ge 0$. In this Section, we express the evaluation map 
$\on{ev}_t\colon R_k(G)\to \C$ for $t=t_\lambda$ in $K$-homology terms. Fix a
$T_{k+\cox}$-equivariant Morita trivialization
\begin{equation}\label{eq:B}
(p,\D)\colon (T,\A^{(k+\cox)}|_T)\da (\pt,\C) 
\end{equation}
as in Proposition \ref{prop:restrict}. It gives a push-forward map: 
\begin{equation}\label{eq:pushtopt}
K_0^{T_{k+\cox}}(T,\A^{(k+\cox)}|_T) 
%\xra{\cong} K_0^{T_{k+\cox}}(T)
\to K_0^{T_{k+\cox}}(\pt).\end{equation}
The Localization Theorem \ref{th:locali} (applied to the conjugation 
action of $t\in T_{k+\cox}$ on $G$) gives an isomorphism
\begin{equation}\label{eq:2fold}
(\iota_T)_*\colon K_0^{T_{k+\cox}}(T,\A^{(k+\cox)}|_T)_{(t)}\to K_0^{T_{k+\cox}}(G,\A^{(k+\cox)})_{(t)}\end{equation}
Let $K_0^G(G,\A^{(k+\cox)})\to 
K_0^{T_{k+\cox}}(G,\A^{(k+\cox)})_{(t)}$
be the map restricting the action, followed by the quotient map to the
localized module. Composing with the inverse of \eqref{eq:2fold} we
obtain a map
\begin{equation}\label{eq:1fold}
K_0^G(G,\A^{(k+\cox)})\to K_0^{T_{k+\cox}}(T,\A^{(k+\cox)}|_T)_{(t)}
\end{equation}
\begin{proposition}
For any $t\in T_{k+\cox}\cap G^{\on{reg}}$, the evaluation map $\on{ev}_t\colon
R_k(G)\to \C$ factors as a composition, 
\begin{equation}\label{eq:compost}
 R_k(G)\cong K_0^G(G,\A^{(k+\cox)}) \to K_0^{T_{k+\cox}}(T,\A^{(k+\cox)}|_T)_{(t)}
\to K_0^{T_{k+\cox}}(\pt)_{(t)}\xra{\on{ev}_t} \C
\end{equation}
Here the first map is \eqref{eq:1fold} and the second map is 
\eqref{eq:pushtopt}, localized at $t$.
\end{proposition}
Thus, while \eqref{eq:pushtopt} depends on the choice of Morita
trivialization \eqref{eq:B} (which was defined only up to a line bundle over $T$), 
the composition of maps \eqref{eq:compost} is independent of that
choice.
\begin{proof}
We have a commutative diagram, with \eqref{eq:compost} as the upper row, 
\[ \begin{CD} 
K_0^G(G,\A^{(k+\cox)}) @>>{\eqref{eq:1fold}}> K_0^{T_{k+\cox}}(T,\A^{(k+\cox)}|_T)_{(t)}
@>>{{\eqref{eq:pushtopt}}}> K_0^{T_{k+\cox}}(\pt)_{(t)}@>>{\on{ev}_t}> \C\\
@AAA @AA{(B)}A @. @.\\
{K_0^G(\pt)} @>>{(A)}> {K_0^{T_{k+\cox}}(\pt)_{(t)}} @.  @. 
\end{CD}\]
Here (A) is the map restricting the
action to $T_{k+\cox}$ followed by the quotient map to the localized
module, while (B) is the push-forward under inclusion of
$\pt$ as the group unit in $T$, using the canonical $G$-equivariant
trivialization of $\A|_{\{e\}}$. The composition of the maps (B) and
\eqref{eq:pushtopt} is the identity, since we have chosen the Morita
trivialization of $\A^{(k+\cox)}|_T$ to extend the canonical
trivialization at $e$. Hence the composition of maps (A), (B), 
\eqref{eq:pushtopt}, $\on{ev}_t$ is just the evaluation 
map $\on{ev}_t\colon R(G)\to \C$. Hence the 
quotient map 
$R(G)\cong K_0^G(\pt)\to R_k(G)\cong K_0^G(G,\A^{(k+\cox)})$
followed by \eqref{eq:compost} is $\on{ev}_t\colon R(G)\to \C$. 
\end{proof}

\subsection{Localization formula}\label{subsec:loc}
Suppose $M$ is an even-dimensional compact $G$-manifold, and $\Phi\colon M\to G$ is  
a smooth $G$-equivariant map covered by a Morita morphism 
\begin{equation}\label{eq:given}
 (\Phi,\F)\colon (M,\Cl(TM))\da (G,\A^{(k+\cox)}).
 \end{equation}
We have seen how such a Morita morphism arises for pre-quantized q-Hamiltonian $G$-spaces, 
but there are some other examples of interest. If $X\subset M$ is a $T_{k+\cox}$-invariant submanifold with $\Phi(X)\subset T$, 
one obtains by composition with \eqref{eq:B}
a $T_{k+\cox}$-equivariant Morita morphism 
\begin{equation}\label{eq:D}
(X,\Cl(TM)|_X) \da  (T,\A^{(k+\cox)}|_T) \da (\pt,\C),
\end{equation}
hence the bundle $TM|_X$ acquires a $T_{k+\cox}$-equivariant
$\Spin_c$-structure.

The morphism \eqref{eq:given} defines a map 
$K_0^G(M,\Cl(TM))\to  K_0^G(G,\A^{(k+\cox)})=R_k(G)$, and as in the q-Hamiltonian case we denote by 
$\ca{Q}(M,E)\in R_k(G)$ the image of $[E]\in K^0_G(M)=K_0^G(M,\Cl(TM))$ under this map, for any 
$G$-equivariant complex vector bundle $E$. We are interested in the numbers 
\[ \ca{Q}(M,E)(t_\lambda),\ \ \lambda\in\Lambda^*_k\] 
For $t=t_\lambda$, every component  $F\subset M^{t_\lambda}$ of the fixed point set is a $T$-invariant 
submanifold with $\Phi(F)\subset G^{t_\lambda}=T$. Hence the restriction $TM|_F$ carries a
$T_{k+\cox}$-equivariant $\Spin_c$-structure. Let 
\[ \ca{Q}(\nu_F,E)(t_\lambda)\]
be the corresponding `Atiyah-Segal-Singer' fixed point contribution. 

It may be characterized as follows. Let $U_F\cong \nu_F$ be a tubular neighborhood of $F$. The $\Spin_c$-structure on 
$TM|_F$ extends to a $T_{k+\cox}$-equivariant 
$\Spin_c$-structure on $U_F$, defining a $\Spin_c$-Dirac operator $\dirac_{U_F}$ 
with coefficients in $E|_{U_F}$. We let $\ca{Q}(\nu_F,E)(t_\lambda)$ be \emph{its} fixed point contribution, i.e.~ 
the image of its class 
$[\dirac_{U_F}]\in K_0^{T_{k+\cox}}(U_F)_{(t)}$ under 
the composition 
\[ K_0^{T_{k+\cox}}(U_F)_{(t)}\to K_0^{T_{k+\cox}}(F)_{(t)}\to K_0^{T_{k+\cox}}(\pt)_{(t)}\to \C\]
given by localization to $F$, followed by push-forward to $\pt$, followed by evaluation. 
The number $\ca{Q}(\nu_F,E)(t_\lambda)$ is given by a well-known cohomological formula, which we 
recall below.

\begin{theorem}\label{th:fix}
For all $\lambda\in\Lambda^*_k$, the number $\ca{Q}(M,E)(t_\lambda)$ is a sum over connected components
$F\subset M^{t_\lambda}$:
\[ \ca{Q}(M,E)(t_\lambda)=\sum_F \ca{Q}(\nu_F,E)(t_\lambda).\]
\end{theorem}
\begin{proof}
We may compute $\ca{Q}(M,E)(t)$ for $t=t_\lambda$ using the commutative diagram, 
\[ \begin{CD} K_0^G(M,\Cl(TM))_{(t)} @>>{\Phi_*}>
  K_0^G(G,\A^{(k+\cox)})_{(t)}@>>{\on{ev}_t}> \C\\
  @VVV @VVV @AA{\on{ev}_t}A\\
  \bigoplus_F K_0^{T_{k+\cox}}(F)_{(t)} @>>{\bigoplus_F (\Phi|_F)_*}>
  K_0^{T_{k+\cox}}(T)_{(t)}@>>>
  K_0^{T_{k+\cox}}(\pt)_{(t)}
\end{CD}\]
where the first two vertical maps are given by  Atiyah-Segal 
localization. Extend the $\Spin_c$-structure on $TM|_F$ to a $T_{k+\cox}$-equivariant  
$\Spin_c$-structure on a tubular neighborhood $U_F\cong \nu_F$, as above.  The $F$-summand of the left vertical map 
factors as a composition, 
\[ K_0^G(M,\Cl(TM))_{(t)}\to K_0^{T_{k+\cox}}(U_F,\Cl(TU_F))_{(t)}\cong K_0^{T_{k+\cox}}(U_F)_{(t)}\to 
K_0^{T_{k+\cox}}(F).\]
The first map takes $[E]\cap [M]$ to $[E|_{U_F}]\cap [U_F]\in K_0^{T_{k+\cox}}(U_F,\Cl(TU_F))_{(t)}$. The isomorphism 
with $K_0^{T_{k+\cox}}(U_F)_{(t)}$, defined by the $\Spin_c$-structure, takes $[E|_{U_F}]\cap [U_F]$ to the class 
$[ \dirac_{U_F}]$
of the Dirac operator with coefficients in $E$,  see Proposition \ref{prop:roe}. The map to $K_0^{T_{k+\cox}}(F)$ is its localization; if we compose further with the map to $K_0^{T_{k+\cox}}(\pt)$ 
we get the fixed point contributions $\ca{Q}(\nu_F,E)(t_\lambda)$. 

%Using the identification \eqref{eq:KK} to express K-homology in terms of K-theory, 
%we may write the left vertical map as 
%
%\[  K^0_{G,cp}(TM,\pi_M^*\Cl(TM))_{(t)}\to  K^0_{T_{k+\cox},cp}(TF)_{(t)}\]
%
%given as pull-back to $TF$ followed by multiplication by the inverse of $\pi_F^*\lambda^{-1}(\nu_F\otimes\C)$. 
%This map factors through the K-theory of $TU$, where $U$ is a closed $T_{k+\cox}$-equivariant 
%tubular neighborhood of $F$ (thus $U$ is a manifold with boundary). 
%We may extend the $\Spin_c$-structure on $TM|_F$ to a $T_{k+\cox}$-equivariant $\Spin_c$-structure on $U$, and 
%are thus simply computing the fixed point contributions for the $\Spin_c$-Dirac operator $\dirac_E$, defined over $U$. 
%The latter are given by  the usual Atiyah-Segal-Singer formulas. 
%
\end{proof}

To work out some examples, we will need the cohomological form of the fixed point 
contributions.  These are given by integrals
\[ \ca{Q}(\nu_F,E)(t_\lambda)=\int_F \f{\wh{A}(F)\, 
\Ch(E|_F,t_\lambda)\ \Ch(\ca{L}_F,t_\lambda)^{1/2}}{\D_\R(\nu_F,t_\lambda)}.\]
The terms in this integral are de Rham cohomology classes in $H^\bullet(F,\R)$, defined as follows.
(See e.g.~ \cite{be:he} for further details.) 
\begin{enumerate}
\item[(i)]
The class
$\wh{A}(F)$ is  given on the level of forms as 
${\det}^{-1/2}_\R(j(\f{1}{2\pi}R_{TF}))$, where $R_{TF}\in \Om^2(F,\End(TF))$ is the curvature form 
of a Riemannian connection, and $j(z)=\f{\on{sinh}(z/2)}{z/2}$. 
\item[(ii)]
$\on{Ch}(E|_F,t)$ is given on the level of forms as 
$\on{tr}_\C\big(A_E(t) \exp(\f{1}{2\pi} R_E)\big)$ where $A_E(t)\in \Gamma(\on{U}(E|_F))$ is the action of $t$,  
$R_E\in \Om^2(F,\End(E|_F))$ is the curvature form of a Hermitian connection on $E|_F$, and $\on{tr}_\C$ denotes the trace of a complex endomorphism. 
\item[(iii)]
The class $D_\R(\nu_F,t)$ is given on the level of 
forms as 
\[ (\sqrt{-1})^{\hh \on{rk}(\nu_F)}{\det}^{1/2}_\R\Big(1-A_F(t)^{-1}\exp(\f{1}{2\pi} R_F)\Big)\]
where $A_F(t)\in \Gamma(\on{O}(\nu_F))$ is the action of $t$ on $\nu_F$,  
$R_F\in \Om^2(F,\mf{o}(\nu_F))$ is the curvature form, and using the positive square root of the 
determinant. 
\item[(iv)]
Finally let $\L_F\to F$ be the line bundle associated to the $\Spin_c$-structure on $TM|_F$. 
The class $\on{Ch}(\L_F,t)$ is defined as in (ii), but we will need to specify a 
square root.  Letting $c_1(\ca{L}_F)$ be the 
first Chern form, and $\zeta_F(t)\in\U(1)$ the (locally constant) action of $t$ on $\L_F$
we put
\[\Ch(\ca{L}_F,t)^{1/2}=\zeta_F(t)^{1/2}\exp(\hh c_1(\L_F)).\] 
To fix the sign of the square root of the phase factor, choose a $t$-invariant orthogonal complex structure $J$
on $T_xM$, for any $x\in F$. Thus $T_xM$ becomes a complex vector space, and $t$ acts as a unitary transformation 
$A_x(t)\in \on{U}(T_xM)$. The given $\Spin_c$-structure on $T_xM$ and the $\Spin_c$-structure defined by $J$ differ 
by a Hermitian line $L$, on which $t$ acts by some phase factor $\kappa(t)\in\U(1)$. 
We put  
\[ \zeta_F(t)^{1/2}=\kappa(t)\ {\det}_\C(A_x(t)^{1/2})\]
where $A_x(t)^{1/2}\in \on{U}(T_xM)$ is the unique square root for which all eigenvalues are 
of the form $e^{i\phi}$ with $0\le \phi<\pi$. (See \cite{al:fi}.)
\end{enumerate}

\subsection{Normalized form of the fixed point contributions}
As mentioned before, the localization procedure described above depends on the choice of the Morita trivialization 
\eqref{eq:B}. Changing \eqref{eq:B} by a line bundle $N\to T$ replaces $\hh c_1(\L_F)$ with $\hh c_1(\L_F)+\Phi_F^* c_1(N)$. 
While the individual fixed point contributions may be affected by this change, their sum always computes 
$\ca{Q}(M,E)(t)$. Since the sum is a priori a polynomial in the coefficients of $c_1(N)$ 
(after choosing a basis of $H^2(T,\Z)$), we see that we may in fact replace $\hh c_1(\L_F)$ with $\hh c_1(\L_F)-\Phi_F^*x$, for any $x\in H^2(T,\R)$. By the discussion of Section \ref{subsec:res}, the choice of Morita trivialization \eqref{eq:B}
determines a class in $H^3(G,T,\Z)$. Let $x$ be its first component under 
the inclusion $H^3(G,T,\Z)\to H^2(T,\R)\oplus H^3(G,\R)$. 
Twisting  by a line bundle $N$ replaces $x$ with $x+c_1(N)$; hence $\hh c_1(\L_F)-x$ is independent of the choice of \eqref{eq:B}. By the results from Section \ref{subsec:res} the class $x$ lies in $\hh H^2(T,\Z)$. It is thus of the form 
$\hh c_1(R)$ for a line bundle $R\to T$, and we conclude that $\L_F\otimes\Phi_F^* R^{-1}$ is independent of the 
choice of \eqref{eq:B}, up to isomorphism. 

Taking this to be our new definition of $\L_F$, we obtain a `normalized form' of the fixed point contributions, 
which no longer depends on any choices. Notice however that the new $\L_F$ need not be the line bundle associated to a $\Spin_c$-structure on $TM|_F$, in general. 

It is this `normalized' form of the fixed point 
contributions $\chi(\nu_F,t)$ that identifies our localization formula with the expressions in 
\cite[Theorem 4.3]{al:fi}. One of the advantages 
of this canonical form is the invariance under the Weyl group: Given $w\in W$ and a fixed point component 
$F\subset M^t$ of $t=t_\lambda$, its image $w.F$ is again a fixed point component (possibly equal to $F$), 
and one has 
\[ \chi(\nu_{w.F},t)=\chi(\nu_F,w^{-1}.t).\]

\section{Examples}
In this Section, we will work out the fixed point contributions in various examples. 
The quantization of conjugacy classes $\Co\subset G$ and the double $D(G)$, 
discussed in Section \ref{subsec:condo}, were already considered 
in \cite{al:fi} from the perspective of Hamiltonian loop group actions. 
They are included here due to their significance, and since the approach in 
this paper is more direct.  The computation 
for the q-Hamiltonian $\SU(n)$-space $S^{2n}$ 
in Section \ref{subsec:sphere} is new. 

We refer to Appendix \ref{app:fus} for some of the Lie-theoretic notions used in this Section.

\subsection{Conjugacy classes, double}\label{subsec:condo}
Let $G$ be a compact, simply connected Lie group, and $k$ a given integral level, with $k_i\ge 0$.  We identify the set of conjugacy classes in $G$ with the elements of the fundamental alcove $\Alc\subset\t$, 
thus $\xi\in\Alc$ labels the conjugacy class $\Co$ of $a=\exp(\xi)$. 
Recall that $\Co$ is pre-quantized at level $k$ if and only if 
$B_k^\flat(\xi)\in\Lambda^*_k$. On the other hand, $R_k(G)$ has an additive basis 
$\tau_\mu,\ \mu\in\Lambda^*_k$.  
\begin{proposition}
Let $\xi\in \Alc$ with $\mu:=B_k^\flat(\xi)\in\Lambda^*_k$, and denote by 
$\Co=G.\exp(\xi)$ the conjugacy class labeled by $\xi$. Then
\[ \ca{Q}(\Co)=\tau_\mu.\]
\end{proposition}
\begin{proof}
For any open face $\sig\subset\Alc$ of the alcove, let $G_\sig$ be the centralizer of points in $\exp(\sig)\subset T$. It is a connected subgroup of 
$G$ containing $T$. Let $W_\sig\subset W$ be the Weyl groups and 
$\mf{R}_\sig\subset \mf{R}$ the set of roots of $G_\sig\subset G$. 
Define a set $\mf{R}_{\sig,+}$ of \emph{positive roots} to consist of 
all $\alpha\in\mf{R}_\sig$ that are non-negative on the subset $\Alc-\zeta$, 
for any $\zeta\in\sig$. The corresponding Weyl chamber is the cone over 
$\Alc-\zeta$. The set of positive roots $\mf{R}_{\sig,+},\mf{R}_+$
determine $T$-invariant complex structure on $\g_\sig/\t,\ \g/\t$, 
hence also $\Spin_c$-structures. By the exact sequence 
$0\to \g_\sig/\t\to \g/\t\to \g/\g_\sig\to 0$, the 
space $\g/\g_\sig$ inherits a $T$-equivariant $\Spin_c$-structure.
(In general, this $\Spin_c$-structure cannot be made $G_\sig$-equivariant, hence does not defined a $\Spin_c$-structure on $G/G_\sig$.) For $t\in T$, the 
phase factor and denominator factor for the `fixed point contribution' 
from $0$,  corresponding to the $\Spin_c$-structure on $\g/\t$, are 
\[ \zeta(t)^{1/2}=t^\rho,\ \ \ca{D}_\R(\g/\t,t)=J(t)\]
where $\rho$ is the half-sum of positive roots and $J(t)=\sum_{w\in W}(-1)^{l(w)}t^{w\rho}$ is the Weyl denominator. 
The contributions for $\g_\sig/\t$ are given similarly, with 
$\rho$ and $J$ replaced by  
\[ \rho_\sig=\hh \sum_{\alpha\in\mf{R}_{\sig,+}}\alpha,\ \ \ 
J_\sig(t)=\sum_{v\in W_\sig}(-1)^{l(v)}t^{v\rho_\sig}
%=\prod_{\alpha\in\mf{R}_{\sig,+}}\ (t^{\alpha/2}-t^{-\alpha/2})
\]
respectively,  
and the contributions for $\g/\g_\sig$ are the quotients of those for $\g/\t$, $\g/\g_\sig$. (Here $\rho_\sig$ need not be a weight, in general. Instead, we define $t^{\rho_\sig}$ as the square root of the determinant of the action of $t$ on $\g/\g_\sig$, cf.~ Section \ref{subsec:loc} (iv).) Now take $\sig$ to be the open face containing the given element $\xi$. 
Then $\Co=G/G_\sig$ as a $G$-manifold, and the fixed point set 
of $t=t_\lambda,\lambda\in\Lambda^*_k$ is 
$\Co^{t}=W/W_\sig$ (identified with the 
$W$-orbit of $\exp(\xi)$). Thus 
\[ \ca{Q}(\Co)(t)=\f{1}{|W_\sig|} \sum_{w\in W}  \chi(\nu_F,w^{-1}t),\]
where $F=\{\exp\xi\}$. 
The $\Spin_c$-structure on $T\Co|_F=T_{\exp\xi}\Co\cong \g/\g_\sig$  
is the restriction of the $T$-equivariant 
$\Spin_c$-structure on $\g/\g_\sig$ described above, twisted by the 
1-dimensional $T$-representation $\C_\mu$ of weight $\mu$. 
(See \cite{me:con} for a detailed discussion.)
Hence, the fixed point contribution of $F=\{\exp\xi\}$ is 
\[ 
\chi(\nu_F,t)= t^{\mu+\rho-\rho_\sig} \f{J_\sig(t)}{J(t)}. 
\]
We now use that $\mu\in B_k^\flat(\sig)$ and $\rho-\rho_\sig\in
B_{\cox}^\flat(\sig)$, hence $\mu+\rho-\rho_\sig\in
B_{k+\cox}^\flat(\sig)$.  The action of $v\in W_\sig$ on
$\exp(\sig)\subset T$ is trivial, hence the action on $\sig\subset\t$
is a shift by some vector in $\Lambda$. It follows that
$v(\mu+\rho-\rho_\sig)-(\mu+\rho-\rho_\sig)\in B_{k+\cox}(\Lambda)$,
hence
$t^{\mu+\rho-\rho_\sig+v\rho_\sig}=t^{v(\mu+\rho)}$. This gives 
\[ \chi(\nu_F,t)=\f{1}{J(t)}\sum_{v\in W_\sig}(-1)^{l(v)} t^{v(\mu+\rho)}.\]
Using the anti-invariance 
$J(w^{-1}t)=(-1)^{l(w)} J(t)$ and the Weyl character formula, 
we hence obtain 
\[ \sum_{w\in W} \chi(\nu_F,w^{-1}t)=
\f{1}{J(t)} \sum_{w\in W} \sum_{v\in W_\sig}(-1)^{l(wv)} t^{wv(\mu+\rho)}
=|W_\sig|\,\chi_\mu(t)=
|W_\sig|\,\tau_\mu(t).\]
\end{proof}

Let $N_{\mu,\nu,\lambda}=(\tau_\mu\tau_\nu\tau_\lambda)^G\in\Z$ 
be the structure constants of $R_k(G)$. Thus 
$\tau_\mu\tau_\nu=\sum_\lambda N_{\mu,\nu,\lambda} \tau_\lambda^*$. 
\begin{proposition}[Double]
The level $(k,k)$ quantization of the q-Hamiltonian $G\times G$-space 
$\mathbf{D}(G)\cong M(\Sig_0^2)$ is given by 
\[ \ca{Q}(\mathbf{D}(G))=\sum_{\mu\in \Lambda^*_k} \tau_\mu\otimes \tau_\mu^*\]
as an element of $R_k(G)\otimes R_k(G)$. The quantization of the fused double $D(G)$ is thus 
\[ \ca{Q}(D(G))=\sum_{\mu,\nu\in \Lambda^*_k}
N_{\mu,*\mu,*\nu}\ \tau_\nu
\in R_k(G)\]
\end{proposition}
\begin{proof}
We may assume that $G$ is simple. Recall that the $G\times G$-action on $\mathbf{D}(G)=G\times G$ is $(g_1,g_2).(a,b)
=(g_1 a g_2^{-1},\ g_2 b g_1^{-1})$. We hence see that for $\lambda,\lambda'\in \Lambda^*_k$, the
fixed point set of $t_{\lambda,\lambda'}=(t_\lambda,t_{\lambda'})$ is trivial unless $\lambda=\lambda'$, in which case it is 
$F=T\times T$. In particular, $\wh{A}(F)=1$ in cohomology. The normal bundle $\nu_F$ is a trivial bundle $\g/\t\times\g/\t$; the factor $D_\R(\nu_F,t_\lambda)$ is just $|J(t_\lambda)|^2$. Hence the fixed point contribution is 
\[\ca{Q}(\nu_F,t_\lambda,t_\lambda)=|J(t_\lambda)|^{-2}\zeta_F(t_\lambda)^{1/2}\int_F e^{\hh c_1(\L_F)}.\]
One has $\zeta_F(t_\lambda)^{1/2}=1\in\U(1)$, since it is constant along $F$, 
and equals $1$ at $(e,e)\in F$. To compute the integral of $e^{\hh c_1(\L_F)}$, note that the Morita trivialization $(M,\C)\da (G,\A^{2(k+\cox)})$ obtained by squaring the Morita isomorphism $(M,\Cl(TM))\da (G,\A^{(k+\cox)})$ serves as a pre-quantum line bundle at level $2(k+\cox)$. In particular, $c_1(\L_F)=2(\cox+k)\om_F^{(1)}$ in cohomology, where 
$\om_F^{(1)}$ is the symplectic form on $F$ corresponding to the level $k=1$. It follows that
\[ \int_F e^{\hh c_1(\L_F)}=(k+\cox)^{\dim T}\vol(F)=|T_{k+\cox}|.\]
(Cf.~ \cite[Proposition 5.2]{al:fi}.) That is, 
\[ \ca{Q}(\mathbf{D}(G))(t_\lambda,t_{\lambda'})=\f{|T_{k+\cox}|}{|J(t_\lambda)|^2}\delta_{\lambda,\lambda'}.\]
The Proposition now follows using the orthogonality relations (Proposition \eqref{prop:orth}) for level $k$ characters.
By fusing the two $G$-factors, it follows that the quantization of the fused double $D(G)=M(\Sig_1^1)$ is 
$\ca{Q}(D(G))=\sum_{\mu\in \Lambda^*_k}\tau_\mu^* \tau_\mu$. 
\end{proof}

More generally the quantization of $M(\Sig_h^r)$ of the moduli space for a surface of genus $h$ with 
$r$ boundary components is obtained by fusing $h$ copies of $M(\Sig_1^1)$ with $r-1$ copies of 
$M(\Sig_0^2)$. For example, 
\[ \begin{split}
\ca{Q}(M(\Sig_0^3))&=\sum_{\mu,\nu} \tau_\mu \otimes\tau_\nu \otimes \tau_\mu^*\tau_\nu^*\\
&=\sum_{\mu,\nu,\lambda} N_{\mu,\nu,\lambda}\ \tau_\mu \otimes \tau_\nu \otimes \tau_\lambda.
\end{split}\]
(Here we used that $N_{*\mu,*\nu,*\lambda}=N_{\mu,\nu,\lambda}$.) Alternatively, one may verify this formula by direct 
application of the localization formula.

\subsection{The sphere $S^{2n}$}\label{subsec:sphere}
We will next calculate the level $k$ quantization of the sphere $S^{2n}$, viewed as a q-Hamiltonian 
$\SU(n)$-space. Let $\varpi_1,\ldots,\varpi_{n-1}$ be the fundamental weights of $\SU(n)$. Thus 
$\varpi_i$ is the dominant weight of the (irreducible) representation of $\SU(n)$ on $\wedge^i \C^n$.  
For the proof of the following result, we recall the following description of the 
fusion ideal \cite{bou:pre,gep:fus} 
\begin{equation}\label{eq:theideal}
I_k(\SU(n))=\l \chi_{(k+1)\varpi_1},\ldots,\chi_{(k+n-1)\varpi_1} \r,\end{equation}
the ideal in $R(\SU(n))$ generated by the characters inside the brackets. 
\begin{proposition}
The level $k$ quantization of the q-Hamiltonian $\SU(n)$-space $S^{2n}$ is given by 
\[ \ca{Q}(S^{2n})=\sum_{i=0}^{k} \tau_{i\varpi_1}.\] 
\end{proposition}
\begin{proof}
We need to review the construction of the q-Hamiltonian structure on $M=S^{2n}$. 
Let $\Phi_0\colon \C^n\to \mf{su}(n)^*$ be the moment map for the $\SU(n)$-action on 
$\C^n\cong \R^{2n}$. Its moment polyhedron (i.e. the intersection of the moment map image with 
the positive Weyl chamber) is the ray spanned by $\varpi_1$. Let $U_0\subset \C^n$ be the open ball, 
consisting of all $z\in\C^n$ with $||\Phi_0(z)||<k||\varpi_1||$. Then $\exp$ is regular on 
$B_k^\sharp(\Phi_0(U))$, hence we may `exponentiate' to obtain a q-Hamiltonian 
$\SU(n)$-space $(U_+,\om_+,\Phi_+)$. Let $U_-=U_+$ with 2-form $\om_-=\om_+$ and moment map 
$\Phi_-=c\Phi_+$, where $c$ is the central element obtained by exponentiating the vertex 
$\varpi_1^\vee$ of the alcove. Then $(U_\pm,\om_\pm,\Phi_\pm)$ glue to form 
$(S^{2n},\om,\Phi)$. The moment polytope of $S^{2n}$ is the edge $\{t\varpi_1^\vee|\ t\in [0,1]\}\subset \Alc$
of the alcove. (That is, every element in $\Phi(S^{2n})$ is conjugate to $\exp(t\varpi_1^\vee)$ for a unique 
$t\in [0,1]$.) Consider now the fixed point contributions of $t_\lambda,\ \lambda\in\Lambda^*_k$. 
For the action on $U_+$, the only fixed point is $\exp(0)=e$. Writing $t_\lambda$ as a diagonal matrix 
with entries $z_1,\ldots,z_n$ down the diagonal, the contribution reads
\[ \prod_{j=1}^n (1-z_i^{-1})^{-1}.\]
Insert a parameter $u$ with $|u|<1$, and expand into a power series: 
\[ \prod_{j=1}^n (1-u z_i^{-1})^{-1}=\sum_{m=0}^\infty u^m \sum_{r_1+\ldots+r_n=m} z_1^{-r_1}\cdots z_n^{-r_n}\]
The coefficient of $u^m$ is the character of the representation on the $m$-th symmetric power 
$\on{Sym}^m(\C^n)$. This is an irreducible representation of weight $m\varpi_1$. Hence 
\[ \prod_{j=1}^n (1-u z_i^{-1})^{-1}=\sum_{m=0}^\infty u^m \chi_{m\varpi_1}(t_\lambda).\]
For the action on $U_-\cong U_0$, the fixed point is $c=c\exp(0)=\exp(\varpi_1^\vee)$, and the 
weights are of course the same. However, the $\Spin_c$-structure on $T_cM$ (constructed as in the general 
theory by means of the $T_{k+n}$-equivariant Morita trivialization of the DD bundle $\A|_T$, cf.~ \eqref{eq:B}) 
does not coincide with the $\Spin_c$-structure on $T_0U_0=T_0\C^n$. The Morita trivialization was normalized 
so that it coincides with the $\SU(n)$-equivariant Morita trivialization at $e$, but at $c$ it does 
not coincide with the unique $\SU(n)$-equivariant Morita trivialization at $c$: The two differ by a character 
$t\mapsto t^{(k+n)\varpi_1}$ of $T_{k+n}$. Secondly, the orientation on $T_c U_-$ coming from the orientation on $S^{2n}$ 
is opposite to the orientation coming from $T_0U_0=\C^n$. This changes the $\Z_2$-grading on the spinor module, 
and accounts for a minus sign in the fixed point contribution: 
\[ -  t^{(k+n)\varpi_1} \prod_{j=1}^n (1-z_i^{-1})^{-1}.\]
Once again, we insert a variable $u$ with $|u|<1$, and expand: 
\[ - u^{k+n} t^{(k+n)\varpi_1} \prod_{j=1}^n (1-u z_i^{-1})^{-1}
=- u^{k+n} t^{(k+n)\varpi_1} \sum_{m=0}^\infty u^m \chi_{m\varpi_1}(t_\lambda).\] 
For $\mu\in B_{k+\cox}^\flat(\Lambda)$, we have $t_\lambda^{w\mu}=t_\lambda^\mu$ for all $w\in W$. 

Using the Weyl character formula, we may use this to deduce $t_\lambda^{(k+n)\varpi_1} \chi_{m\varpi_1}(t_\lambda)=\chi_{(m+k+n)\varpi_1}(t_\lambda)$. Hence, adding the two expressions, the coefficients of $u^m$ with $m>k+n$ cancel, and we obtain a polynomial in $u$: 
\[ \sum_{m=0}^{k+n-1} u^m \chi_{m\varpi_1}(t_\lambda).\]
Letting $u\to 1$, this shows 
\[ \ca{Q}(S^{2n})(t_\lambda)= \sum_{m=0}^{k+n-1} \chi_{m\varpi_1}(t_\lambda).\]
By \eqref{eq:theideal}, $\chi_{m\varpi_1}(t_\lambda)=0$ for $m=k+1,\ldots,k+n-1$.  On the other hand, 
the weights $m\varpi_1$ for $m\le k$ are in $\Lambda^*_k$, thus $\chi_{m\varpi_1}(t_\lambda)=\tau_{m\varpi_1}(t_\lambda)$. 
\end{proof}

\begin{appendix}

\section{Equivariant de Rham theory}\label{app:eq}
Let $G$ be a Lie group, with Lie algebra $\g$.  A $G$-action on a
manifold $M$ is a group homomorphism $\A\colon G\to \on{Diff}(M)$ for
which the action map $G\times M\to M,\ (g,m)\mapsto \A(g)(m)$ is
smooth.  A $\g$-action on $M$ is a Lie algebra homomorphism $\A\colon \g\to
\mf{X}(M)$ into the Lie algebra of vector fields, such that the map
$\g\times M\to TM,\ (\xi,m)\mapsto \A(\xi)|_m$ is smooth. Every
$G$-action defines a $\g$-action by $\A(\xi)=\f{d}{d
  t}|_{t=0}\A(\exp(-t\xi))^*$, where we interpret vector fields as
derivations of $C^\infty(M)$. 

We denote by $(\Om_G^\bullet(M),\d_G)$ the equivariant de Rham complex
(or Cartan complex), i.e.~
\[ \Om_G^k(M)=\bigoplus_{2i+j=k}(S^i\g^*\otimes\Om(M))^G,\ \ (\d_G\beta)(\xi)
=\d(\beta(\xi))-\iota(\A(\xi))\beta(\xi).\]
(Here elements of $S\g^*$ are interpreted as polynomials on $\g$.) For
$G$ compact, its cohomology $H(\Om_G(M),\d_G)$ coincides with the
equivariant cohomology group $H_G(M,\R)$.

\section{Relative cohomology}\label{app:rel}
We recall the following construction from homological algebra.
Suppose $R^\bullet,S^\bullet$ are two cochain complexes, and
$f^\bullet\colon S^\bullet\to R^\bullet$ is a cochain map. The \emph{algebraic mapping cone}
is the cochain complex
\[ \cone^k(f):=R^{k-1}\oplus S^k,\ \ d(x,y)=(f(y)-\d x,\d y).\]
Its cohomology is called the \emph{relative cohomology} of the cochain map
$f$. The short exact sequence of cochain complexes 
$0\to R^{\bullet-1}\to\cone^\bullet(f)\to S^\bullet\to 0$ gives rise to a long exact sequence of cohomology groups, 
\[ \cdots H^{k-1}(R)\to H^k(\cone(f))\to H^k(S)\to H^k(R)\to \cdots,\]
here the connecting homomorpism $H^k(S)\to H^k(R)$ is just the map
defined by $f$. Given a commutative diagram 
\[ \begin{CD} S @>>{f}> R\\
@VVV @VVV \\
S' @>>{f'}> R'
\end{CD}\]
of cochain complexes and cochain maps, one obtains a cochain map 
$\cone(f)\to \cone(f')$, and a homomorphism of the corresponding 
long exact sequences in cohomology. In particular, if the vertical
maps 
$S\to S'$ and $R\to R'$ are quasi-isomorphisms, then by the 5-lemma 
$\cone(f)\to \cone(f')$ is a quasi-isomorphism also. 

If $R,S$ are the complexes of singular cochains on (reasonable) spaces $X,Y$, and
$f=\Phi^*$ for some map $\Phi\colon X\to Y$, then the relative cohomology 
group $H(f)$ coincides with the cohomology group 
$H(\Phi):=H(\cone(\Phi))$ of the \emph{topological mapping cone}. 
Here $\cone(\Phi)$ is obtained by gluing the cone $\cone(X)\supset X$ with $Y$, using
$\Phi$ as a gluing map. If $X,Y$ are manifolds and $\Phi$ is a smooth
map, one can compute the group $H(\Phi,\R)$ using differential forms, as the
cohomology of the algebraic mapping cone for the map $\Phi^*\colon
\Om(Y)\to \Om(X)$. A class in $H(\Phi,\R)$ (or the cocycle representing
it) is called \emph{integral} if it lies in the image of the map
$H(\Phi,\Z)\to H(\Phi,\R)$. The integrality of de Rham cocycles is
characterized in terms of the pairing with relative cycles: A cocycle
$(\alpha,\beta) \in \Om^k(\Phi)$ is integral if and only if for every
(smooth) singular chains $\Sigma \in C_{k-1}^{\on{sm}}(X),\ \Upsilon \in
C_k^{\on{sm}}(Y)$ with $\partial\Sigma=0$,
$\partial\Upsilon=\Phi(\Sigma)$, one has
\[\int_\Sigma\alpha-\int_{\Upsilon}\beta\in \Z.\]
The definitions generalize to $G$-spaces and $G$-maps in an obvious way. 
If $\Phi$ is a $G$-map, we define $H_G(\Phi):=H(\Phi_G)$, where
$\Phi_G\colon X_G\to Y_G$ is the map of the Borel constructions.
(Note that $\cone \Phi_G\not=(\cone \Phi)_G$, in general. Thus
$H_G(\Phi)$ is not to be confused with the equivariant cohomology of
$\cone(\Phi)$.)
If $\Phi$ is a smooth $G$-map between $G$-manifolds, then $H_G(\Phi,\R)$  
is computed using the algebraic mapping cone 
\[\Om_G^k(\Phi)=\Om^{k-1}_G(X)\oplus \Om^k_G(Y).\] 

\section{$K$-homology}\label{app:khom}
In this Section, we give a brief summary of some aspects of $K$-homology, following Kasparov's 
approach.  Some aspects of bivariant K-theory will be needed as well. For details, we refer to Kasparov's papers 
\cite{ka:op,ka:eq} and the monograph \cite{hig:ana}. 
\subsection{Kasparov modules}
Let $G$ be a compact Lie group, and $\mathsf{A}$ a (separable) $\Z_2$-graded $G$-$C^*$-algebra. 
Elements of the $K$-homology groups $K^0_G(\mathsf{A})$ are represented by 
\emph{Kasparov modules}
$(\H,\pi,F)$, where $\H$ is a $\Z_2$-graded $G$-Hilbert space with a $G$-equivariant 
even $*$-homomorphism $\pi\colon \mathsf{A}\to \mathbb{B}(\H)$, and 
$F\in \mathbb{B}(\H)$ 
is a $G$-invariant odd bounded operator, such that for all $a\in \mathsf{A}$ 
\[ [\pi(a),F],\ \ (F^2-I)\pi(a),\ \ (F^*-F)\pi(a)\]
are in the ideal $\mathbb{K}(\H)$ of compact operators. Such triples form a semi-group under direct sum, and 
a semi-group $K^0_G(\mathsf{A})$ is defined as its quotient by an equivalence relation of `homotopy'. In fact, the semigroup 
$K^0_G(\mathsf{A})$ is a group. $K$-homology is contravariant in $\mathsf{A}$, and 
the definition $K^1_G(\mathsf{A})=K^0_G(\mathsf{A}\otimes\Cl(\R))$ makes 
$K^\bullet_G(\mathsf{A})$ into a $\Z_2$-graded homology theory on $G$-$C^*$-algebras. 

If $(\H,\pi,F)$ is a Kasparov module over $\mathsf{A}$, then $(\ol{\H},\pi^*,F^*)$ is 
a Kasparov module over $\mathsf{A}^{\on{op}}$. Here $\ol{\H}$ is viewed as the dual space to $\H$, and 
$\pi^*(a)=\pi(a)^*$. This defines a group homomorphism $*\colon K^\bullet_G(\mathsf{A})\to  
K^\bullet_G(\mathsf{A}^{\on{op}})$ with $*^2=1$.  

We are mainly interested in the case $\mathsf{A}=C_0(X)$ (the continuous functions on a $G$-space $X$, vanishing at infinity), or more generally 
$\mathsf{A}=\Gamma_0(X,\A(X))$ (the continuous sections of a $G$-DD bundle over $X$, vanishing at infinity).
\footnote{Throughout, we assume that $X$ is a `reasonable' space, e.g. a 
countable $G$-CW complex.} In these cases, we write 
$K_i^G(X):=K^i_G(C_0(X))$ resp. $K_i^G(X,\A)=K^i_G(\Gamma_0(X,\A))$. These groups behave similar to Borel-Moore homology (homology with locally finite supports).

One has $K_1^G(\C)=0$, while $K_0^G(\C)=R(G)$, the representation ring of $G$. The identification
takes $(\H,\pi,F)$ to the equivariant Fredholm index of $F$. The ring structure on $R(G)$  
is given by the `cross product' on K-homology (cf.~ \cite[page 265]{hig:ana}), 
and the involution $*$ on $R(G)$ is the involution $*$ on $K$-homology.  More generally, 
using the cross product, the equivariant $K$-homology groups $K_\bullet(\mathsf{A})$ are modules over $R(G)$.

\subsection{Morita morphisms}
The K-homology groups $K^\bullet_G(\mathsf{A})$ are functorial 
not only with respect to $C^*$-algebra homomorphisms, but more generally
with respect to \emph{Morita morphisms}. Suppose $\mathsf{E}$ is a 
$\Z_2$-graded $G$-equivariant Morita $\mathsf{A}_2-\mathsf{A}_1$ bimodule in the sense of Rieffel, where  $\mathsf{A}_1,\ \mathsf{A}_2$ are $\Z_2$-graded $G$-$C^*$-algebras. 
Thus $\mathsf{E}$ is a $G$-equivariant 
Hilbert module over $\mathsf{A}_1$ with an even $G$-equivariant  
$*$-homomorphism $\mathsf{A}_2\to \mathbb{B}(\mathsf{E})$. (See e.g. 
\cite{lan:hi} for an introduction to Hilbert modules.) Together with 
$F:=0\in \mathbb{B}(\mathsf{E})$ one obtains an element $[\mathsf{E}]$ 
of the bivariant $K$-group $KK_G(\mathsf{A}_2,\mathsf{A}_1)$. (The latter is defined similar to 
the K-homology groups, but with Hilbert modules in place of Hilbert spaces.)
Kasparov product with this element gives a group homomorphism 
$K^\bullet_G(\mathsf{A}_1)\to K^\bullet_G(\mathsf{A}_2)$. 

Morita morphisms  $(\Phi,\E)\colon (X_1,\A_1)\da (X_2,\A_2)$ of 
$G$-DD bundles, as defined in Section \ref{subsubsec:mormor},  give rise to Morita morphisms of the $C^*$-algebras 
of sections: Let $\mathsf{E}=\Gamma_0(X_1,\E),\ \mathsf{A}_i=\Gamma_0(X_i,\A_i),\ i=1,2$ be the continuous sections vanishing at infinity. The $\Phi^*\A_2-\A_1$-bimodule structure of $\E$ defines a $\mathsf{A}_2-\mathsf{A}_1$-bimodule structure of $\mathsf{E}$. There is a fiberwise 
`$\A_1$-valued inner product'
$\E\times\E\da \A_1$, given in local trivializations by the map
$\K(\H_1,\H_2)\times \K(\H_1,\H_2)\to \K(\H_1),\ \ (a,b)\mapsto a^* b$. 
Passing to sections, one obtains a $\mathsf{A}_1$-valued inner product on $\mathsf{E}$, giving $\mathsf{E}$ the structure of a Hilbert module over $\mathsf{A}_1$, and indeed of a Morita  $\mathsf{A}_2-\mathsf{A}_1$ bimodule. As a consequence, any Morita morphism $(\Phi,\E)$ as above induces a group homomorphism 
$K_\bullet^G(X_1,\A_1)\to K_\bullet^G(X_2,\A_2)$. 

\subsection{The K-homology fundamental class}
An important example of a $K$-homology class is the \emph{fundamental class}, constructed by Kasparov in 
\cite[$\mathsection$ 4]{ka:eq}, where it is called the \emph{Dirac element}.  
Let $M$ be a Riemannian $G$-manifold, and $\Cl(TM)$ its bundle of Clifford algebras. Let $\d$ be the de Rham differential on $\Gamma^\infty(M,\wedge T^*M)$, $\d^*$ its dual, and $D=\d\oplus \d^*$. Let $\H=\Gamma_{L^2}(M,\wedge T^*M)$ (the square integrable sections), and $F=D/(1+D^2)^{1/2}\in \mathbb{B}(\H)$. The operator $D$ is the Dirac operator for the Clifford 
module $\wedge T^*M\cong \Cl(TM)$, defined using the $\Cl(TM)$-action by multiplication from the left. This 
bundle carries a second $\Cl(TM)$-module structure defined by multiplication from the right. It defines an even  homomorphism $\pi\colon \Gamma_0(M,\Cl(TM))\to \mathbb{B}(\H)$. The triple 
$(\Gamma_{L^2}(M,\wedge T^*M),\pi,F)$
defines a class $[M]\in K^0_G(\Gamma_0(\Cl(TM))$. 
The $\Z_2$-graded $G$-DD bundle $\wt{\Cl}(TM):=\Cl(TM\oplus \R^n)$ (where $n=\dim M$) plays the role of the orientation bundle in $K$-theory, in the sense that a Morita trivialization of this bundle amounts to a $K$-orientation, i.e.~ $\Spin_c$-structure. Thus $[M]\in K_n(M,\wt{\Cl}(TM))$ is the 
analogue of the fundamental class in $H_n^{B.M.}(M,o_M)$, the Borel-Moore homology with coefficients in the orientation bundle $o_M=\wedge^n TM$. 

Suppose now that $\dim M$ is even, and that $M$ carries a $\Spin_c$-structure $\ca{S}$. 
The $\Spin_c$-Dirac operator $\dirac$ for $\ca{S}$ defines a K-homology class $[\dirac]\in K_0(M)$, given by the 
triple $(\H,\pi,F)$ where $\H=\Gamma_{L^2}(\ca{S})$ with the natural action $\pi$ of 
$C(M)$, and $F=\dirac/(1+\dirac^2)^{1/2}$. On the other hand, the spinor module $\ca{S}$ 
defines a $G$-equivariant Morita isomorphism 
\[ (\on{id},{\ca{S}})\colon
(M,\Cl(TM))\da (M,\C)\] (here $\ca{S}$ is regarded as a right module for $\Cl(TM)$), hence 
an isomorphism 
\begin{equation}\label{eq:morisom1} K_0^G(M,\Cl(TM))\to K_0^G(M).\end{equation}
\begin{proposition}\label{prop:roe}
Suppose $M$ is an even-dimensional 
Riemannian $G$-manifold with equivariant $\Spin_c$-structure. 
The isomorphism \eqref{eq:morisom1} takes the fundamental class $[M]$ to the class $[\dirac]$ of the $\Spin_c$-Dirac operator. More generally, if $E$ is a $G$-equivariant Hermitian vector bundle over $M$, 
it takes the class $[E]\cap [M]$ to the class $[\dirac^E]$ of the $\Spin_c$-Dirac operator with coefficients in $E$.
\end{proposition}
(If $\dim M$ is odd, one has a similar statement, but with a $\Spin_c$-Dirac operator defining an odd 
$K$-homology class.) The following proof was explained to me by John Roe. 
\begin{proof}
We outline the argument for $E=\C$, the general case is a straightforward extension. 
For convenience, we work with the Morita morphism  $(\on{id},\ol{\ca{S}})\colon
(M,\C)\da (M,\Cl(TM))$ `inverse' to $(\id,\ol{\ca{S}})$. The isomorphism
$K_0^G(M)\to K_0^G(M,\Cl(TM))$ is given by Kasparov product with the class in bivariant $K$-theory $KK_G(\Gamma_0(\Cl(TM)),C_0(M))$, defined by the 
triple $(\Gamma(\ol{\ca{S}}),\pi,0)$ where $\Gamma(\ol{\ca{S}})$ is viewed as a Hilbert $C_0(M)$-module with an action 
$\pi$ of $\Gamma(\Cl(TM))$. This Kasparov product amounts to `coupling' to 
the vector bundle $\ol{\ca{S}}$. That is, it takes $[\dirac]$ to $[\dirac_{\ol{\ca{S}}}]$, the $\Spin_c$-Dirac operator 
with coefficients in the $\Cl(TM)$-module $\ol{\ca{S}}$. But $\ca{S}\otimes \ol{\ca{S}}=\Cl(TM)\cong \wedge T^*M$ as a $\Cl(TM)\otimes \Cl(TM)$-module. That is, $\dirac_{\ol{\ca{S}}}$ is the de Rham Dirac operator. 
\end{proof}

\subsection{Thom isomorphisms}
Kasparov's version of the \emph{Thom isomorphism theorem} in K-theory asserts that for any $G$-equivariant Euclidean vector bundle $\pi\colon V\to X$, the $C^*$-algebras $C_0(V)$ and $\Gamma_0(\Cl(V))$ are $KK_G$-equivalent. That is, there 
are canonically given elements $\alpha\in KK_G(C_0(V),\Gamma_0(\Cl(V)))$ and $\beta\in KK_G(C_0(V),\Gamma_0(\Cl(V)))$
that are inverses under the Kasparov product. More generally, we have: 
\begin{lemma}[Thom isomorphism in twisted $K$-theory]
Suppose $\A\to X$ is a $G$-DD bundle. Then the $G-C^*$-algebras 
$\Gamma_0(\A\otimes\Cl(V))$ and $\Gamma_0(\pi^*\A)$ are $KK_G$-equivalent. 
\end{lemma}
I am grateful to Nigel Higson for help with the following argument.
\begin{proof} 
In \cite{ka:eq} Kasparov considered a refinement 
$\ca{R}KK_G^X(\mathsf{A},\mathsf{B})$ of the bivariant $K$-groups, with arguments a $G$-equivariant space $X$, and $G-C_0(X)$ $C^*$-algebras $\mathsf{A},\mathsf{B}$. Elements of this group are represented by triples 
$(\H,\pi,\,F)$ as in the $KK_G$-theory, with a 
requirement that the $C_0(X)$-actions on $\mathsf{A},\mathsf{B}$ are compatible with the 
bimodule action on the Hilbert module 
$\H$. In \cite[Section 5]{ka:op}, Kasparov shows that the 
$C^*$-algebras $C_0(\R^n)$ and $\Cl(\R^n)$ are 
$KK_{\on{O}(n)}$-equivalent. Letting $P$ be the $\on{O}(n)$-frame bundle 
of $V$, it follows that $C_0(P\times \R^n)$ and $C_0(P,\Cl(\R^n))$ are 
$\ca{R}KK_{G\times \on{O}(n)}^P$-equivalent. 
Using the Descent Theorem \cite[Theorem 3.4]{ka:eq} one may take a quotient by 
$\on{O}(n)$, and obtains that 
$C_0(V)$ and $\Gamma_0(\Cl(V))$ are $\ca{R}KK_G^X$-equivalent. 
Tensoring over $C_0(X)$ with $\Gamma_0(\A)$, it follows that 
$C_0(V)\otimes_{C_0(X)}\Gamma_0(\A)=\Gamma_0(\pi^*\A)$ and $\Gamma_0(\Cl(V))\otimes_{C_0(X)}\Gamma_0(\A)=
\Gamma_0(\A\otimes\Cl(V))$ are $\ca{R}KK_G^X$-equivalent, 
hence also $KK_G$-equivalent. 
\end{proof}

\begin{remark}
The Thom isomorphism in $K$-theory, for vector bundles that are not necessarily $K$-oriented, 
was one of the original motivations for Donovan-Karoubi's \cite{don:gr} definition of 
$K$-theory with local coefficients. Carey-Wang \cite{car:th} have proved a Thom isomorphism in twisted K-theory using 
`bundle gerbe modules'. \end{remark}

The Thom isomorphism may be used to define wrong-way maps in twisted $K$-homology. 
For example, suppose that $\iota\colon N\subset M$ is a $G$-invariant embedded submanifold 
of even codimension, and $\A\to M$ is a DD bundle. Let $\nu\to N$ be the normal bundle. 
Then we obtain a map 
\[ \iota^!\colon K_\bullet^G(M,\A)\to K_\bullet^G(N,\A|_N\otimes \Cl(\nu))\] 
by composing the restriction $K_\bullet^G(M,\A)\to K_\bullet^G(U,\A|_U)$ to a tubular neighborhood 
$U$ of $N$ with a Thom isomorphism for $\nu\cong U$. As in \cite{ros:co}, one defines the twisted equivariant $K$-theory with compact supports $K^\bullet_{G,cp}(X,\A)$ 
as the equivariant $K$-theory of $\Gamma_0(X,\A)$. If $M$ is an even-dimensional manifold, one has the following Poincar\'{e} duality statement: 
\begin{proposition} (\cite{tu:twi1}, see also \cite{bro:dbr,ech:kk}.)
Suppose $M$ is a compact even-dimensional manifold\footnote{We assume $n=\dim M$ is even so that $\Cl(TM)$ is a $G$-DD bundle. 
The odd-dimensional case may be handled by working with $\wt{\Cl}(TM)=\Cl(TM\oplus \R^n)$.}, and $\A\to M$ is a $G$-DD bundle.
Then there are canonical Poincar\'{e} duality isomorphisms
\[ K_\bullet^G(M,\A)\cong K^\bullet_G(M,\Cl(TM)\otimes \A^{\on{op}}),\ \ K^\bullet_G(M,\A)\cong K_\bullet^G(M,\Cl(TM)\otimes \A^{\on{op}}).\]
\end{proposition}

\section{The level $k$ fusion ring}\label{app:fus}
We will view the representation ring $R(G)$ of a compact Lie group $G$ as the 
subring of
$C^\infty(G)$, spanned by the characters of finite-dimensional
representations of $G$. Suppose $G$ is simply connected, with a decomposition $G=G_1\times\cdots\times 
G_N$ into simple factors, and let $k\in \Z^N$, with $k_i\ge 0$. 
The level $k$ fusion ring (Verlinde algebra) $R_k(G)$ is the quotient 
of $R(G)$ by an ideal of characters vanishing on a certain finite collection of conjugacy classes. 
To describe these conjugacy classes, fix a maximal torus $T$
and a closed positive Weyl chamber $\t_+\subset \t$. Let $\Alc$ be the 
unique closed alcove contained in $\t_+$ and containing $0$.  The alcove parametrizes conjugacy 
classes in $G$, in the sense that any element of $G$ is conjugate to $\exp(\xi)$ for a unique element 
$\xi\in\Alc$. The regular elements of $G$ correspond to points in the interior of $\Alc$. 

Let $\Lambda\subset \t$ be the integral lattice (kernel of the exponential map
$\t\to T$), and let $\Lambda^*\subset \t^*$ be its dual, the (real) weight lattice. 
We have the usual identifications $\Lambda=\on{Hom}(\U(1),T)$ and 
$\Lambda^*=\on{Hom}(T,\U(1))$. 
Put $\Alc^*_k=B_k^\flat(\Alc)\subset \t^*$, and let 
\[ \Lambda^*_+=\Lambda^*\cap \t^*_+,\ \ \ \Lambda^*_k=\Lambda^*\cap \Alc^*_k\]
be the set of dominant weights and \emph{level $k$ weights}, respectively. A dominant weight 
resp.~ level $k$ weight is \emph{regular} if it lies in the interior 
of $\t^*_+$, resp.~ of $\Alc^*_k$. Letting $\rho\in \Lambda^*$ be the half-sum of positive roots, 
and $\cox=(\cox_1,\ldots,\cox_N)$, where $\cox_i$ is the dual Coxeter number of $G_i$, 
one has
\[ \Lambda^*_{+,\reg}=\Lambda^*_++\rho,\ \ \ \Lambda^*_{k+\cox,\reg}=\Lambda^*_k+\rho.\]
The map $B_{k+\cox}^\sharp=(B_{k+\cox}^\flat)^{-1}$ takes $\Lambda^*_{k+\cox,\reg}$ to the interior of 
$\Alc$; thus 
\begin{equation}\label{eq:tlambda}
 t_\lambda=\exp(B_{k+\cox}^\sharp(\lambda+\rho)),\ \ \lambda\in\Lambda^*_k\end{equation}
are regular elements of $G$. 
\begin{definition}\label{def:fus}
The  \emph{level $k$ fusion ideal} $I_k(G)\subset R(G)$ is the ideal of characters vanishing 
on the conjugacy classes of the elements \eqref{eq:tlambda}. The \emph{level $k$ fusion ring} is the quotient ring 
\[ R_k(G)=R(G)/I_k(G).\]
\end{definition}
Clearly, $R_k(G)$ is the direct product of the fusion rings $R_{k_i}(G_i)$ for the simple factors. The involution of $R(G)$ given by complex conjugation $\lambda\mapsto \lambda^*$ of characters 
preserves $I_k(G)$, hence it descends to an involution $\tau\mapsto \tau^*$ of $R_k(G)$. 
The evaluation of characters at $t_\lambda$ descends to the fusion ring: 
\[ \on{ev}_{t_\lambda}\colon R_k(G)\to \C,\ \tau\mapsto \tau(t_\lambda).\]
\begin{remark}
      $R_k(G)$ may also be interpreted as the fusion ring of level $k$
      projective representations of the loop group $LG$. However, we
      will not need this viewpoint in what follows.
\end{remark}

\begin{remark}
For $l=(l_1,\ldots,l_N)$ with $l_i\in \Z$, the 
bilinear form $B_l$ takes on integer values on $\Lambda$. Hence the map 
$B_l^\flat\colon \t\to \t^*$ takes $\Lambda$ into $\Lambda^*$. Assuming that all $l_i>0$, 
so that $B_l^\sharp=(B_l^\flat)^{-1}$ is defined, we may thus define a finite subgroup 
\[ T_l=B_l^\sharp(\Lambda^*)/\Lambda\subset T=\t/\Lambda.\]
The elements $t_\lambda$ defined in \eqref{eq:tlambda} lie in $T_{k+\cox}$; conversely, any regular element in
$T_{k+\cox}$ is $W$-conjugate to a unique $t_\lambda$.  
\end{remark}

For any dominant weight $\mu\in \Lambda^*_+$, let $\chi_\mu\in R(G)$ denote the character of the irreducible 
representation of highest weight $\mu$. Given a level $k$ weight $\mu\in\Lambda^*_k$, we denote by 
$\tau_\mu\in R_k(G)$ the image of $\chi_\mu$ under the quotient map $R(G)\to R_k(G)$. 
Let $J\in R(T)$ be the Weyl denominator: 
\begin{equation}\label{eq:weyl}
J(t)=\sum_{w\in W} (-1)^{l(w)} t^{w\rho}.
\end{equation}

The $\chi_\mu,\ \mu\in\lambda^*_+$  form an additive basis of $R(G)$. Similarly:
\begin{proposition}\label{prop:orth}
The elements $\tau_\mu,\ \mu\in\Lambda^*_k$ form a $\Z$-basis of the level $k$ fusion ring 
$R_k(G)$. (In particular, $R_0(G)=\Z$.) One has the orthogonality relations 
\[ \begin{split}
\sum_{\lambda\in\Lambda^*_k} |J(t_\lambda)|^2
\tau_\mu(t_\lambda)\tau_{\mu'}^*(t_\lambda)&=|T_{k+\cox}|\
\delta_{\mu,\mu'},\ \ \ \ \mu,\mu'\in \Lambda^*_k,\\
\sum_{\mu\in\Lambda^*_k} |J(t_\lambda)|^2
\tau_\mu(t_\lambda)\tau_{\mu}^*(t_{\lambda'})&=|T_{k+\cox}|\
\delta_{\lambda,\lambda'},\ \ \ \ \lambda,\lambda'\in \Lambda^*_k.
\end{split}\]
\end{proposition}
These formula are consequences of the Weyl character formula
and `finite Fourier transform'. See e.g. \cite{be:co,al:fi}. Using the orthogonality relations, any 
$\tau\in R_k(G)$ may be recovered from the values $\tau(t_\lambda)$ as 
$\tau=\sum_{\lambda\in\Lambda^*_k} N(\mu)\tau_\mu$ 
where 
\[ N(\mu)=\f{1}{|T_{k+\cox}|}\sum_{\lambda\in\Lambda^*_k} |J(t_\lambda)|^2\,\tau(t_\lambda)\tau^*_\mu(t_\lambda).
 \]

The quotient map $\pi\colon R(G)\to R_k(G)$ has the following description in the basis. 
Let $\Waff=\Lambda\rtimes W$ be the affine Weyl group. Its shifted 
action at level $k$ on $\t^*$, denoted $\mu\mapsto w\bullet_k \mu$, 
is generated by reflections across the affine hyperplanes
\[ H_{\alpha,m}^{(k)}=\{\mu\in\t^*|\ \l\alpha,\  B_{k+\cox}^\sharp(\mu+\rho)\r=m\} \]
for roots $\alpha$ and integers $m\in\Z$. 
%boundary hyperplanes of $B_{k+\cox}^\flat(\Alc)-\rho$. 
That is, if $w\in W\subset \Waff$ then $w\bullet_k\mu =w(\mu+\rho)-\rho$,
while for $w=\xi\in\Lambda$ we have $w\bullet_k \mu=\mu-B_{k+\cox}(\xi)$. In terms of this action, 
\[ \pi(\chi_\mu)=\begin{cases}(-1)^{l(w)}\tau_{w\bullet_k\mu}
&\mbox{ if }\ w\bullet_k\mu\in \Lambda^*_k,
\\
0&\mbox{ if }\ \exists w\not=1\colon w\bullet_k\mu=\mu. 
\end{cases}\]
The two cases are exclusive, i.e.~ $\Waff\bullet\Lambda^*_k$ is precisely the set of 
weights whose stabilizer under the shifted action at level $k$ is trivial. 
\end{appendix}

\def\cprime{$'$} \def\polhk#1{\setbox0=\hbox{#1}{\ooalign{\hidewidth
  \lower1.5ex\hbox{`}\hidewidth\crcr\unhbox0}}} \def\cprime{$'$}
  \def\cprime{$'$} \def\polhk#1{\setbox0=\hbox{#1}{\ooalign{\hidewidth
  \lower1.5ex\hbox{`}\hidewidth\crcr\unhbox0}}} \def\cprime{$'$}
  \def\cprime{$'$}
\providecommand{\bysame}{\leavevmode\hbox to3em{\hrulefill}\thinspace}
\providecommand{\MR}{\relax\ifhmode\unskip\space\fi MR }
\providecommand{\MRhref}[2]{%
  \href{http://www.ams.org/mathscinet-getitem?mr=#1}{#2}
}
\providecommand{\href}[2]{#2}

\vskip1in \end{document}